%% file: random_attractors_for_singular_SPDE.tex
\documentclass[12pt,leqno]{article}
\usepackage[active]{srcltx}
\usepackage{amsmath, amsthm, amsfonts, amssymb, color}
\usepackage[matrix,arrow]{xy}
\usepackage{mathrsfs}
\usepackage{enumerate}
\usepackage{amsopn} 
\usepackage{bbm} 
\usepackage{ulem} 
\usepackage{hyperref}
\input{preamble}

\usepackage{mdef}

\title{\bf Random attractors for singular stochastic partial differential equations.}

\author{
\textbf{Benjamin Gess}
\thanks{Supported by DFG-Internationales Graduiertenkolleg “Stochastics and Real World Models”, the SFB-701 and the BiBoS-Research Center. \vskip5pt \textbf{Acknowledgments:} The author would like to thank Michael R\"ockner for valuable discussions and comments.}  \\
\small{Faculty of Mathematics, University of Bielefeld, Germany} \\
\small{bgess@math.uni-bielefeld.de}
}

\begin{document}
 
\maketitle

\begin{abstract}
  The existence of random attractors for singular stochastic partial differential equations (SPDE) perturbed by general additive noise is proven. The drift is assumed only to satisfy the standard assumptions of the variational approach to SPDE with compact embeddings in the Gelfand triple and singular coercivity. For ergodic, monotone, contractive random dynamical systems it is proven that the attractor consists of a single random point. In case of real, linear multiplicative noise finite time extinction is obtained. Applications include stochastic generalized fast diffusion equations and stochastic generalized singular $p$-Laplace equations perturbed by L\'evy noise with jump measure having finite first and second moments. 
\end{abstract}

\noindent {\bf 2000 Mathematics Subject Classification AMS}: Primary: 37L55, 60H15. Secondary: 37L30. 

\noindent {\bf Key words }: stochastic partial differential equations, stochastic fast diffusion equation, stochastic $p$-Laplace equation, random dynamical system, random attractor.

\bigskip
\maketitle


\section{Introduction}
The dynamical behaviour of random systems induced by stochastic (partial) differential equations has attained much interest in recent years. Especially the analysis of the long-time behaviour of such systems by means of the existence of random attractors has been intensively discussed since the foundational work in \cite{CDF97,CF94,S92}. However, for quasilinear SPDE the existence of random attractors could so far only be proven for degenerate drifts occurring for example in porous media and degenerate $p$-Laplace equations. The case of singular equations requires different techniques and is solved in this paper for the first time. The existence of attractors to certain singular PDE such as singular reaction-diffusion equations has been an open problem even in the deterministic case. For related results on deterministic, degenerate PDE which are partially complemented by this paper we refer to \cite{CCD99,CG01,CG03,CD00,CR10,LYZ10,TY03,W10,YSZ07} and references therein.

Until recently the existence of random attractors could only be shown for concrete examples of SPDE of semilinear type, i.e.\ of the form
  \[ dX_t = (AX_t + F(X_t))dt + B(X_t)dW_t, \]
with linear main part $A$. The first approach to a truly quasilinear stochastic equation has been presented in \cite{BGLR10} by proving the existence of random attractors for generalized stochastic porous media equations. This specific example has then been recovered in \cite{GLR11} (at least for more regular noise) as an application of a first general result providing the existence of random attractors for a class of SPDE perturbed by general noise, i.e.\ for equations of the form
  \[ dX_t = A(X_t)dt + dN_t. \]
In that paper superlinear/degenerate drifts, i.e.\ satisfying $\Vbk{A(v),v} \ge c\|v\|_V^\a$ with $\a \ge 2$ as well as an additional approximative coercivity condition have been considered. Regarding applications, the superlinear case corresponds to degenerate parabolicity as occurring for example in porous media equations and degenerate $p$-Laplace equations.

The results mentioned above are complemented in several ways by the present paper. First, we consider SPDE with \textit{singular} drift (i.e.\ satisfying $\Vbk{A(v),v} \ge c\|v\|_V^\a = \frac{c}{\|v\|_V^{2-\a}}\|v\|^2$ with $1< \a < 2$) of the form
\begin{equation}\label{ra_s:eqn:spde}
  dX_t = A(t,X_t)dt + dN_t.
\end{equation}
Such equations are called singular since the coercivity coefficient $\frac{c}{\|v\|_V^{2-\a}}$ is singular when $\|v\|_V$ approaches $0$. Second, we will not require the additional approximative coercivity condition used in \cite{GLR11} and thereby we are able to allow much rougher noise $N_t$. Third, we present an approach that allows to combine the knowledge about the existence of a random pullback attractor and the ergodicity of the associated Markovian semigroup to prove that the random attractor consists of a single random point, which in turn is a globally stable equilibrium of the RDS. 

For deterministic equations it is well known that the dynamical behaviour of systems induced by singular equations differs strongly from the one produced by superlinear/degenerate drifts. For example, while solutions to porous media equations (PME) decay to $0$ at a polynomial rate (cf.\ \cite{A81}), finite time extinction occurs for solutions to fast diffusion equations (cf.\ \cite{V07} and references therein). Concerning the existence of attractors two main obstacles occur in case of singular drifts. First, it is more difficult to obtain a global control for the solutions or in other words to prove bounded absorption for the associated RDS, since the coercivity coefficient in
 $$\Vbk{A(v),v} \ge \frac{c}{\|v\|_V^{2-\a}} \|v\|_V^2$$
degenerates for large values of $v$. This problem is solved by the present paper by proving new a priori estimates for singular ODE. Second, the regularizing properties of singular equations are weaker than in the degenerate case. For example, consider the stochastic singular $p$-Laplace equation (S$p$LE) with reaction term $G$
  \[ dX_t = \left(div(|\nabla X_t|^{\a-2} \nabla X_t)+G(X_t)\right)dt + dN_t,\quad 1 < \a < 2.\]
In the degenerate ($\a > 2$) and in the singular case ($1 < \a < 2$) solutions take values in $W^{1,\a}_0 \subseteq L^2$. While for degenerate equations this implies regularization into the invariant subspace $W_0^{1,2}$, this fails in the singular case. Therefore, a different technique to obtain attraction by a compact set is needed. Even in the deterministic case the existence of attractors to such singular equations has been an open problem.

Another prominent example of a singular SPDE is the stochastic fast diffusion equation 
  \[ dX_t = \left(\D(|X_t|^{\a-1}sgn(X_t)) + g(t)\right)dt + dN_t,\quad 1 < \a < 2. \tag{SFDE}\]
Among other applications, SFDE are used as models for heat diffusion in plasma and for self-organized criticality \cite{BDPR09-2}. Concerning the theory of self-organized criticality, in particular the convergence of arbitrary initial states to the critical state (which is a key property of systems exhibiting \textit{self-organized} criticality) and therefore the long-time behaviour of the solutions is of importance. Some results for SFDE perturbed by linear multiplicative space-time noise have been given in \cite{BDPR09-2,BR11}. The physically relevant case of additive noise has not yet been considered. Application of our general result to SFDE with additive noise proves the existence of a pullback random attractor consisting of a single random point, or equivalently the existence of a globally stable random equilibrium. Stochastic porous media equations and SFDE have been intensively investigated in recent years (cf.\ e.g.\ \cite{BDPR08,BDPR08-2,BDPR09,DPR04,DPRRW06,G11c,K06,RRW07,RW08} and references therein).

In the applications it is important to consider noise with only small spatial correlations, which corresponds to noise satisfying only low spatial regularity. While in \cite{GLR11} it essentially had to be assumed $\D N_t \in V$ we only require $N_t \in V$ (where $V$ is the Banach space of the Gelfand triple associated to the variational formulation of an SPDE, c.f.\ $(A1)$-$(A4)$ below) by adopting a technique from \cite{CCD99}. In addition, in \cite{GLR11} the drift $A$ was assumed to be weakly coercive in some additional, compactly embedded space $S\subseteq H$ (cf.\ \cite[$(H5)$]{GLR11}). This assumption will not be needed here and thus our results apply to any singular equation of the form \eqref{ra_s:eqn:spde} fitting into the variational framework (i.e. satisfying $(A1)$-$(A4)$ below) as long as the embedding $V \subseteq H$ is compact. This enables us to cover stochastic \textit{generalized} singular $p$-Laplace equations and stochastic \textit{generalized} fast diffusion equations. \\
The main idea is that the variational approach to SPDE is based on a regularizing property of the drift, meaning that solutions take values in the smaller space $V \subseteq H$ for almost all times. We use this property to deduce the compactness of the stochastic flow (i.e.\ $S(t,s;\o)B$ is a compact set for all $B\subseteq H$ bounded), which in turn yields attraction by a compact set as soon as bounded absorption has been shown.

For deterministic dynamical systems $\vp$ on partially ordered spaces $H$ it is well known that a monotonicity (or order-preserving) property of $\vp$ (i.e.\ $\vp(t)x \ge \vp(t)y$, for $x \ge y$) significantly simplifies the dynamics. In some recent work (cf.\ \cite{AC98,CJ10,C02,CS04,S08} and the references therein) such monotonicity properties have been used to study dynamical properties of RDS. In particular, in \cite{CS04} it has been shown that a monotone, ergodic RDS has a weak pullback attractor consisting of a single random point. However, it had to be assumed that the cone $H_+ \subseteq H$ of nonnegative elements of $H$ has nonempty interior, which is not satisfied by many commonly used state spaces, as for example $L^p$ spaces. We prove that for an ergodic, monotone, contractive RDS on a partially ordered space satisfying the existence of upper bounds (i.e.\ for $x,y\in H$ there exists $z \in H$ with $x,y \le z$) the random attractor consists of a single point. In contrast to the assumptions in \cite{CS04}, upper bounds do exist in $L^p$ spaces as well as in $H=(H_0^1(\mcO))^*$ and thus our results can be applied to SFDE as well as to S$p$LE. We emphasize that the standard approach to prove single-valuedness of the random attractor for degenerate equations (cf.\ e.g.\ \cite{BGLR10,GLR11}) does not apply to singular equations since singular drifts do not satisfy the required strong monotonicity conditions.

As concrete examples we consider stochastic generalized fast diffusion equations and stochastic generalized singular $p$-Laplace equations perturbed by additive noise $N_t$ with $N_t$ having stationary increments, c\`adl\`ag paths and sufficiently slow growth. In particular, this includes all L\'evy processes with jump measure having finite first and second moments. The general results will then prove existence and compactness of an associated RDS and the existence of a random pullback attractor. In the case that $N_t$ is an infinite dimensional Brownian Motion, we will further prove that the associated Markovian semigroup is strongly mixing and that the random attractor consists of a single point, hence is a globally stable equilibrium point.

For singular SPDE perturbed by real linear multiplicative noise, i.e.\ of equations of the form
\begin{equation}\label{ra_s:eqn:SPDE_real_mult}
  dX_t = A(t,X_t)dt + \mu X_t \circ d\b_t,
\end{equation}
we show that the long-time behaviour can be described by forward random attractors (and hence weak random attractors), where $\b_t$ is a real-valued Brownian motion and $\circ$ is the Stratonovich stochastic integral. For the notions of forward and weak random attractors see \cite{S02}.  In fact, we will show a lot more, namely finite time extinction, i.e.\ $\mcA(\o)=\{0\}$ is forward-absorbing. This result is related to the model of self-organized criticality as presented in \cite{BDPR09-2}. With our simplified structure of the noise, i.e.\ space-independent noise, we can strengthen the assertion of finite time extinction with non-zero probability proven in \cite{BDPR09-2} to almost sure finite time extinction. We consider a stochastic perturbation in the Stratonovich sense, since the corresponding It\^o noise causes an artificial stabilization of the random dynamics as it has been observed in \cite{CCLR07}.

In Section 1 we recall some basics on stochastic flows, RDS and random attractors. The precise assumptions and main results will be given in Section 2, while their proofs are postponed to Section 4. In Section 3 we present the application of our general results to SFDE and SpLE.

\section{Basics on stochastic flows and RDS}
We now recall the basic framework of stochastic flows, RDS and random attractors. Let $(H,d)$ be a complete separable metric space and $((\O,\mcF,\P),\{\theta_t\}_{t \in \R})$ be a metric dynamical system, i.e. $(t,\o) \mapsto \theta_t(\o)$ is $(\mcB(\R) \otimes \mcF,\mcF)$-measurable, $\theta_0 =$ id, $\theta_{t+s} = \theta_t \circ \theta_s$ and $\theta_t$ is $\P$-preserving, for all $s,t \in \R$.

\begin{definition}[Stochastic Flow]
  A family of maps $S(t,s;\o): H \to H$, $s \le t$ is said to be a stochastic flow, if for all $\o \in \O$ 
  \begin{enumerate}
    \item[(i)] $S(s,s;\o) = id_H$, for all $s \in \R$.
    \item[(ii)] $S(t,s;\o)x = S(t,r;\o)S(r,s;\o)x$, for all $t \ge r \ge s$, $x \in H$.
  \end{enumerate}
  A stochastic flow $S(t,s;\o)x$ is called
  \begin{enumerate}
    \item[(iii)] measurable if the map $\o \to S(t,s;\o)x$ is measurable for all $s \le t, x \in H$
    \item[(iv)] continuous if the map $x \to S(t,s;\o)x$ is continuous for all $s \le t$, $\o \in \O$
    \item[(v)] a cocycle if 
                  \[ S(t,s;\o)x = S(t-s,0;\t_s\o)x,\]
                 for all $x \in H$, $\o \in \O$ and all $t \ge s$.
  \end{enumerate}
\end{definition}

\begin{definition}[Random Dynamical System]
  A measurable map $\vp: \R_+ \times \O \times H \to H$ satisfying
  \begin{enumerate}
    \item $\vp(0,\o) = 0$
    \item $\vp(t+s,\o) = \vp(t,\t_s \o) \circ \vp(s,\o)$, $\forall \o \in \O,\ s,t \ge 0$
  \end{enumerate}
  is called an RDS. If $x \mapsto \vp(t,\o)x$ is continuous for all $t \in \R$, $\o \in \O$ then $\vp$ is a continuous RDS.
\end{definition}

There is a close connection between RDS and cocycle stochastic flows. Let $S(t,s;\o)$ be a cocycle stochastic flow such that $(t,\o,x) \mapsto S(t,0;\o)x$ is measurable. Then $\vp(t,\o) := S(t,0;\o)$ defines an RDS. Vice versa, let $\vp(t,\o)$ be an RDS and define $S(t,s;\o) := \vp(t-s,\t_s \o)$. Then $S(t,s;\o)$ is a measurable cocycle stochastic flow.\\
Having the notions of a stochastic flow and RDS at our disposal we can now consider their long time behaviour. In the following let $S(t,s;\o)x$ be a stochastic flow.

\begin{definition}
  A family $\{D(t,\o)\}_{t \in \R,\ \o \in \O}$ of subsets of $H$ is said to be
  \begin{enumerate}
   \item a random closed set if it is $\P$-a.s.\ closed and $\o \to d(x,D(t,\o))$ is measurable for each $x \in H$, $t \in \R$. In this case we also call $D$ measurable.
   \item right lower-semicontinuous if for each $t \in \R$, $\o \in \O$, $y \in D(t,\o)$ and $t_n \downarrow t$ there is a sequence $y_n \in D(t_n,\o)$ such that $y_n \to y$ or equivalently $d(y,D(t_n,\o))\to 0$.
  \end{enumerate}
\end{definition}

For normed spaces $H$ we define $\|B\|_H := \sup_{b \in B}\|b\|_H$. 
The a priori bound for solutions to singular SPDE given in the proof of Theorem \ref{ra_s:thm:ra_add} will lead to collections of $\o$-dependent sets satisfying the following growth property:
\begin{definition}\label{ra_s:def:universes}
  Let $H$ be a normed space. A family of sets $\{D(t,\o)\}_{t \in \R,\ \o \in \O}$ is said to be of subpolynomial growth of order $\b > 0$ if 
    \[ \lim_{t \to -\infty} \frac{\|D(t,\o)\|_H}{|t|^{\b}} = 0, \  \forall \o \in \O.\] 
\end{definition}

In the following let $\mcD$ be a system of families $\{D(t,\o)\}_{t \in \R, \o \in \O}$ of subsets of $H$.

\begin{definition}[Absorption \& Attraction]\label{ra_s:def:abs}
  A family of sets $\{F(t,\o)\}_{t \in \R,\o \in \O}$ is said to be 
  \begin{enumerate}
   \item $\mcD$-absorbing, if there is a set $\O_0 \subseteq \O$ of full $\P$-measure such that for all $D \in \mcD$, $t \in \R$ and $\o \in \O_0$ there exists an absorption time $s_0 = s_0(\o,D,t)$ such that
     \[ S(t,s;\o)D(s,\o) \subseteq F(t,\o), \text{ for all } s \le s_0.\]
   \item $\mcD$-attracting, if there is a set $\O_0 \subseteq \O$ of full $\P$-measure such that for all $D \in \mcD$, $t \in \R$ and $\o \in \O_0$                                     
      \[ d(S(t,s;\o)D(s,\o),F(t,\o)) \to 0, \text{ for } s \to -\infty.\]
  \end{enumerate}
\end{definition}

\begin{definition}[Compactness \& Asymptotic Compactness]\label{ra_s:def:cmp_flow}
  A stochastic flow $S(t,s;\o)$ is called
  \begin{enumerate}
   \item $\mcD$-asymptotically compact if there is a $\mcD$-attracting family $\{ K(t,\o) \}_{t \in \R,\o \in \O}$ of compact subsets of $H$.
   \item compact if $S(t,s;\o)B$ is a  precompact subset of $H$ for all $t > s$, $\o \in \O$ and each bounded set $B \subseteq H$.
  \end{enumerate}
\end{definition}

By \cite[Lemma 2.1]{CDF97} we know that a continuous stochastic flow $S(t,s;\o)$ is asymptotically compact iff it is compactly attracted by a compact set for each time $t \in \R$ (where the $\P$-zero set on which attraction occurs may depend on $t$). Let $\{D(t,\o)\}_{t\in \R,\ \o \in \O}$ be a family of subsets of $H$. We define the $\O$-limit set by
    \[ \O(D,t;\o) := \bigcap_{r < t} \overline{ \bigcup_{\tau < r} S(t,\tau;\o)D(\tau,\o)}. \]
There are several stochastic generalizations of the deterministic notion of an attractor. For example, pullback attractors, forward attractors, weak attractors and measure attractors. For a comparison of some of these we refer to \cite{S02}. All of these notions coincide with the usual notion of an attractor in the deterministic case. In the sequel we will mainly work with pullback attractors and simply call them random attractors.

\begin{definition}[Random Attractor]
  A family of sets $\{\mcA(t,\o)\}_{t \in \R,\ \o \in \O}$ is called a $\mcD$-random attractor for $S(t,s;\o)$ if it satisfies $\P$-a.s.
  \begin{enumerate}
   \item $\mcA(t,\o)$ is nonempty and compact, for each $t \in \R$.
   \item $\mcA$ is $\mcD$-attracting.
   \item $\mcA$ is invariant under $S(t,s;\o)$, i.e.
         \[ S(t,s;\o)\mcA(s,\o) = \mcA(t,\o),\ \forall s \le t.\]
  \end{enumerate}
\end{definition}

With this definition we can give a sufficient condition for the existence of a random attractor (cf.\ \cite[Theorem 2.1.]{CDF97}). Let $o \in H$ be some arbitrary point in $H$.
\begin{theorem}[Existence of Random Attractors]\label{ra_s:thm:suff_cond_attr}
  Let $S(t,s;\o)$ be a continuous, $\mcD$-asymptotically compact stochastic flow and let $K$ be the corresponding $\mcD$-attracting family of compact subsets of $H$. Then
  $$\mcA(t,\o) := 
    \begin{cases}
      \overline{\bigcup_{D \in\mcD} \O(D,t;\o)}        &, \text{ if } \o \in \O_0 \\
       \{o\}                                           &, \text{ otherwise.}
    \end{cases}$$
  defines a random $\mcD$-attractor for $S(t,s;\o)$ and $\mcA(t,\o) \subseteq K(t,\o) \cap \O(K,t;\o)$ for all $\o \in \O_0$ (where $\O_0$ is as in Definition \ref{ra_s:def:abs}). 
 
  Let now $s \mapsto S(t,s;\o)x$ be right-continuous locally uniformly in $x$ and $S(t,s;\o)x$ be measurable. If either
  \begin{enumerate}
   \item[(i)] there is a countable family $\mcD_0 \subseteq \mcD$ consisting of right lower-semicontinuous random closed sets such that for each $D \in \mcD$, $\o \in \O$ there is a $D_0 \in \mcD_0$ satisfying $D(t,\o) \subseteq D_0(t,\o)$ for all $t \in \R$ small enough.
   \item[(ii)] $K \in \mcD$ and $K$ is a right lower-semicontinuous random closed set,
  \end{enumerate}
  then $\mcA$ is a random closed set. In case of (ii), $\mcA(t,\o)=\O(K,t;\o)$ for all $\o \in \O_0$.

  If $S(t,s;\o)x$ is a cocycle and either (i) holds with $\mcD_0$ consisting of strictly stationary sets or (ii) is satisfied with $K$ being strictly stationary, then $\mcA$ is strictly stationary.
\end{theorem}

If the $\mcD$-random attractor $\mcA$ is contained in $\mcD$ or $\mcA$ is measurable and strictly stationary with $\{C \subseteq H|\ C \text{ compact}\} \subseteq \mcD$, then $\mcA$ is unique (cf.\ \cite{C99}). Moreover, the random $\mcD$-attractor $\mcA$ constructed in Theorem \ref{ra_s:thm:suff_cond_attr} is uniquely determined as the minimal random $\mcD$-attractor.

The dynamical behavior of RDS can be significantly simpler if the RDS preserves a partial order structure on the state space $H$. For example this idea has been used in \cite{AC98,CJ10, C02, CS04, S08}. A closed, convex cone $H_+ \subseteq H$ satisfying $H_+ \cap (-H_+) = \{0\}$ defines a partial order relation on $H$ which is compatible with the vector structure on $H$ by defining $x \le y$ iff $y-x \in H_+$. A cone $H_+$ is said to be solid if it has nonempty interior.

\begin{definition}[Monotone RDS]\label{ra_s:def:mon_RDS}
  Let $S \subseteq H$. An RDS $\vp$ is said to be
   \begin{enumerate}
    \item monotone on $S$ iff for all $x \le y$, $x,y \in S$, $t \ge 0$, $\o \in \O$
      \[ \vp(t,\o)x \le \vp(t,\o)y.\]
      If $S=H$ then $\vp$ is simply called monotone.
    \item contractive iff $t \mapsto \|\vp(t,\o)x-\vp(t,\o)y\|_H$ is non-increasing for all $x,y \in H$, $\o \in \O$.
   \end{enumerate} 
  
\end{definition}

\section{Setup and Main Results}\label{ra_s:sec:main_result}

Let
  $$V\subseteq H\equiv H^*\subseteq V^*$$
be a Gelfand triple, i.e.\ $H$ is a separable Hilbert space and is identified with its dual space $H^*$ by the Riesz isomorphism $i: H\rightarrow H^*$, $V$ is a reflexive  Banach space such that it is continuously and densely embedded into $H$. $_{V^*}\<\cdot,\cdot\>_V$ denotes the dualization between $V$ and its dual space $V^*$. Let $A: \R \times V \times \O \to V^*$ be such that for each $\o \in \O$, $A(\cdot,\cdot,\o): \R \times V \to V^*$ is $(\mcB(\R) \otimes \mcB(V),\mcB(V^*))$-measurable. We extend the mapping $A$ by $0$ to all of $H$ and assume that there are pathwise right-continuous mappings $C_1,C_2: \R \times \O \to \R$, $c: \R \times \O \to \R_+\backslash\{0\}$ and an \textbf{$\alpha \in (1,2)$} (corresponding to the case of singular equations) such that
\begin{enumerate}
 \item [$(A1)$] (Hemicontinuity) For all $v, v_1,v_2 \in V$, $t \in \R$ and $\o \in \O$, the map 
        $$ s \mapsto { }_{V^*} \< A(t,v_1+s v_2;\o),v \>_V$$
      is continuous on $\mathbb{R}$.
 \item [$(A2)$] (Monotonicity) For all $v_1,v_2 \in V, t \in \R, \o \in \O$ \\
      $$2{  }_{V^*}\<A(t,v_1;\o)-A(t,v_2;\o), v_1-v_2 \>_V \le C_2(t,\o) \|v_1-v_2\|_H^2. $$
 \item [$(A3)$] (Coercivity) There is a function $f: \R \times \O \to \R$ such that $f(\cdot,\o) \in L^1_{loc}(\R)$ and
      $$ 2{ }_{V^*}\<A(t,v;\o), v\>_V \le C_1(t,\o) \|v\|_H^2 - c(t,\o) \|v\|_V^\alpha + f(t,\o),$$
      for each $\o \in \O, t \in \R$ and $v \in V$.
 \item[$(A4)$] (Growth) For each $v \in V, \o \in \O$ and $t \in \R$
      $$ \|A(t,v;\o)\|_{V^*}^{\frac{\a}{\a-1}} \le C_1(t,\o) \|v\|_H^2 + C_2(t,\o) \|v\|_V^\alpha + f(t,\o).$$
\end{enumerate}

\begin{remark}\label{ra_s:rmk:reockner_setup}
  The assumptions needed in \cite{PR07} for the unique existence of a probabilistic solution to an SPDE of the form \eqref{ra_s:eqn:spde} perturbed by Wiener noise are slightly more restrictive (at least in case of additive noise). If we require in addition that $A$ is progressively measurable, $c,C_1,C_2$ are non-random, $f$ is adapted and $f \in L^1_{loc}(\R; L^1(\O))$ then there exists a unique variational solution to such an SPDE (cf. \cite[Theorem 4.2.4]{PR07}).
\end{remark}

\subsection{Additive noise}
Let $(\Omega,\mathcal{F},\{\mathcal{F}\}_{t \in \R},\mathbb{P})$ be a filtered probability space. We consider singular SPDE perturbed by general additive noise, i.e.\ equations of the form
\begin{align}\label{ra_s:eqn:SPDE_add}
  dX_t = A(t,X_t) dt + dN_t, 
\end{align}
where $N_t$ is an $\mcF_t$-adapted $V$-valued stochastic process with stationary increments and c\`adl\`ag paths. More precisely, we assume that $(\O,\mcF,\P,\{\theta_t\}_{t \in \R})$ is a metric dynamical system, i.e. $(t,\o) \mapsto \theta_t(\o)$ is $(\mcB(\R) \otimes \mcF,\mcF)$-measurable, $\theta_0 =$ id, $\theta_{t+s} = \theta_t \circ \theta_s$ and $\theta_t$ is $\P$-preserving, for all $s,t \in \R$ and
\begin{enumerate}
    \item [$(S1)$] (Strictly stationary increments) For all $t,s \in \R$, $\o \in \O$:
                    $$N_t(\o)-N_s(\o) = N_{t-s}(\t_s \o)-N_0(\t_s \o).$$
    \item [$(S2)$] (Regularity) $N_t$ has c\`adl\`ag paths.
\end{enumerate}
If $A$ satisfies $(A1)-(A4)$ we will prove the existence and uniqueness of solutions to \eqref{ra_s:eqn:SPDE_add} in the following sense

\begin{definition}\label{ra_s:def:soln_add} 
  An $H$-valued $\{\F_t\}_{t \in [s,\infty)}$-adapted process $\{X_t\}_{t \in [s,\infty)}$ with c\`adl\`ag paths in $H$ is called a solution of \eqref{ra_s:eqn:SPDE_add} if $X_\cdot(\o) \in L^\alpha_{loc}([s,\infty); V) \cap L^2_{loc}([s,\infty);H)$ and
      $$X_t(\o)=x+\int_s^t A(\tau,X_\tau(\o))\ d\tau + N_t(\o)-N_s(\o) $$
  holds for all $t \in [s,\infty)$, $\o \in \O$.
\end{definition}

We will prove that \eqref{ra_s:eqn:SPDE_add} generates a stochastic flow by first transforming the SPDE into a random PDE and then solving this random PDE for each fixed $\o \in \O$. Let $X(t,s;\o)x$ denote a solution to \eqref{ra_s:eqn:SPDE_add} starting in $x$ at time $s$. Define $\td{X}(t,s;\o)x := X(t,s;\o)x - N_t(\o)$. Then
\begin{align*}
  \td{X}(t,s;\o)x = x-N_s(\o) + \int_s^t A\left(r,\td{X}(r,s;\o)x+N_r(\o)\right) dr.
\end{align*}
Thus, we have to solve the following random PDE
\begin{align}\label{ra_s:eqn:transformed_PDE_add}
   Z(t,s;\o)x = x + \int_s^t A_\o(r,Z(r,s;\o)x) dr,
\end{align}
with $A_\o(r,v) := A(r,v + N_r(\o);\o)$. We then define the stochastic flow associated to \eqref{ra_s:eqn:SPDE_add} by
 \[ S(t,s;\o)x := Z(t,s;\o)(x-N_s(\o)) + N_t(\o),\]
so that $S(\cdot,s;\omega)$ satisfies
  $$ S(t,s;\o)x = x + \int_s^t A(S(r,s;\o)x) dr + N_t(\o) - N_s(\o),$$
for each fixed $\o \in \O$ and all $t \ge s$. Hence $S(t,s;\o)x$ solves \eqref{ra_s:eqn:SPDE_add} in the sense of Definition \ref{ra_s:def:soln_add}. 

Due to the time-inhomogeneity of the drift $A$ we cannot expect the stochastic flow to be a cocycle in general. If, however, the drift is strictly stationary, i.e. if the time-inhomogeneity is only due to the randomness of the drift the cocycle property will be obtained. In this case the stochastic flow induces an RDS associated to \eqref{ra_s:eqn:SPDE_add}.

\begin{theorem}[Generation]\label{ra_s:thm:generation_add}
  Assume $(A1)$-$(A4)$ and $(S1)$-$(S2)$. Then, the family of mappings $S(t,s;\o)x$ is a continuous stochastic flow in $H$. In addition, $S(t,s;\o)x$ is c\`adl\`ag in $t$ and right-continuous in $s$ locally uniformly in $x$. If $A$ is $(\mcB(\R) \otimes \mcB(V) \otimes \mcF,\mcB(V^*))$-measurable then $S(t,s;\o)x$ is a measurable stochastic flow. If $A(t,v;\o)$ is strictly stationary, i.e. $A(t,v;\o)=A(0,v;\t_t \o)$ then $S(t,s;\o)x$ is a cocycle and hence $\vp(t,\o) := S(t,0;\o)$ is a continuous RDS. 
\end{theorem}

As pointed out in the introduction, the variational approach to (S)PDE is based on a regularizing property of the drift $A$. This property is expressed via the coercivity assumption $(A3)$, namely 
\begin{equation}\label{ra_s:eqn:tmp-H3}
  2{ }_{V^*}\<A(t,v), v\>_V \le C_1(t) \|v\|_H^2 - c(t) \|v\|_V^\alpha + f(t). 
\end{equation}
Starting with an initial condition $x \in H$ the second term on the right hand side of \eqref{ra_s:eqn:tmp-H3} yields a control of the solution $X(t,s;\o)x$ in $L^\a_{loc}(\R;L^\a(\O;V))$. In particular $X(t,s;\o)x \in V$, $d$t$\otimes \P$ almost surely. If $V \subseteq H$ is compact, we can use this regularizing effect to prove compactness of the stochastic flow $S(t,s;\o)x$. Since our argument will be purely based on the regularizing effect due to the coercivity assumption, no further restrictions on the drift term have to be required.

\begin{enumerate}
  \item [$(A5)$] Assume that the embedding $V \subseteq H$ is compact.
\end{enumerate}

\begin{theorem}[Compactness]\label{ra_s:thm:compactness}
  Assume $(A1)$-$(A5)$ and $(S1)$-$(S2)$. Then $S(t,s;\o)x$ is a compact stochastic flow.
\end{theorem}

In order to prove the existence of a random attractor we need to assume a growth condition on the paths of the noise.
\begin{enumerate}
  \item [$(S3)$] (Growth) There is a subset $\O_0 \subseteq \O$ of full $\P$-measure such that $\|N_t(\o)\|_V = o(|t|^{\frac{1}{2-\a}})$ for $t \to -\infty$ and all $\o \in \O_0$.
\end{enumerate}

Let $\mcD^\a$ denote the system of all families $\{D(t,\o)\}_{t\in\R,\ \o \in \O}$ of sets of subpolynomial growth of order $\frac{1}{2-\a}$ (cf.\ Definition \ref{ra_s:def:universes}) and let $\mcD^b$ be the system of all deterministic bounded sets. Using comparison Lemmata proven in Section \ref{ra_s:ssec:comp} we obtain
\begin{proposition}[Bounded Absorption]\label{ra_s:prop:add_bounded_absorption}
  Assume that $(A1)$-$(A4)$ with $C_1 \equiv 0$ and $c,C_2$ independent of time $t$, $f(\cdot,\o) = o(|\cdot|^\frac{\a}{2-\a})$  c\`adl\`ag in $t$ and that $(S1)$-$(S3)$ are satisfied. Then there is a right lower-semicontinuous family of sets $\{F(t,\o)\}_{t \in \R, \o \in \O} \in \mcD^\a$ that $\mcD^\a$-absorbs the stochastic flow $S(t,s;\o)x$. If $f$ is measurable in $\o$ then $F$ is a random closed set.
\end{proposition}

Combining bounded absorption and compactness of the stochastic flow we conclude
\begin{theorem}[Existence of Random Attractors]\label{ra_s:thm:ra_add}
  Assume that $(A1)$-$(A5)$ with $C_1 \equiv 0$ and $c,C_2$ independent of time $t$, $f(\cdot,\o) = o(|\cdot|^\frac{\a}{2-\a})$ c\`adl\`ag in $t$ and that $(S1)$-$(S3)$ are satisfied. Then $S(t,s;\o)x$ admits a random $\mcD^\a$-attractor $\mcA^\a \in \mcD^\a$. If $A$ and $f$ are measurable then so is $\mcA^\a$. 

  If $A$ is strictly stationary then there is a strictly stationary $\mcD^b$-random attractor $\mcA^b$ that is measurable if $A$ is.
\end{theorem}

We will now introduce a method that allows to prove that the random attractor consists of a single point if the RDS is monotone, contractive and has an associated weak-$*$ mean ergodic Markov semigroup (cf.\ Definition \ref{ra_s:def:weak_erg} below).

We denote by $\mcB(H)$ the set of all Borel measurable subsets of $H$, by $B_b(H)$ (resp.\ $C_b(H)$) the Banach space of all bounded, measurable (resp.\ continuous) functions on $H$ equipped with the supremum norm and by $Lip_b(H)$ the space of all bounded Lipschitz continuous functions on $H$. By $\mcM_1$ we denote the set of all Borel probability measures on $H$. For a semigroup $P_t$ on $B_b(H)$ we define the dual semigroup $P_t^*$ on $\mcM_1$ by $P_t^* \mu(B) := \int_H P_t \mathbbm{1}_B d\mu$, for $B \in \mcB(H)$. A measure $\mu \in \mcM_1$ is said to be invariant for the semigroup $P_t$ if $P_t^* \mu = \mu$, for all $t \ge 0$. For $T > 0$ and $\mu \in \mcM_1$ we define
  \[ Q^T\mu := \frac{1}{T} \int_0^T P_r^*\mu dr \]
and we write $Q^T(x,\cdot)$ for $\mu = \d_x$. Recall
\begin{definition}\label{ra_s:def:weak_erg}
  A semigroup $P_t$ is called weak-$*$ mean ergodic if there exists a measure $\mu \in \mcM_1$ such that
   \[ \text{w-lim}_{T \to \infty} Q^T\nu = \mu, \]
  for all $\nu \in \mcM_1$ where w-lim is the limit with respect to weak convergence on $\mcM_1$. 
\end{definition}

If an RDS preserves a solid partial order structure on the state space $H$ the strong mixing property of the associated Markov semigroup $P_tf(x) = \E[f(\vp(t,\cdot)x)]$ implies the existence of a weak random attractor consisting of a single random point (cf.\ \cite{CS04}). This result is based on the assumption that the cone of nonnegative elements $H_+$ is solid, which is not satisfied by the cone of nonnegative functions $L^p_+$ in $L^p$ spaces. Assuming only the existence of upper bounds with respect to $H_+$ we prove that the pullback random attractor consists of a single random fixed point if the RDS is contractive and monotone. Moreover, we only assume weak-$*$ \textit{mean} ergodicity of the Markovian semigroup not the strong mixing property. We will require 
\begin{enumerate}
  \item [$(H')$] Assume that there is a cone $H_+ \subseteq H$ with induced partial order structure ''$\le$`` and that there is a dense subset $S \subseteq H$ such that $\vp$ is monotone with respect to ''$\le$`` on $S$. Further assume the existence of upper bounds with respect to ''$\le$`` on $S$, i.e.\ that for all $x,y \in S$ there is an upper bound $z \in S$ satisfying $x,y \le z$.  
\end{enumerate}
While the cone of nonnegative functions $L^p_+$ in $L^p$ spaces is not solid, the existence of upper bounds as required in $(H')$ is satisfied by $L^p_+$. Moreover, for  $H = (H_0^1(\mcO))^*$ the existence of an upper bound $z \in H$ for any two $x,y \in H$ is not clear, while it is obvious as soon as $x, y \in H \cap L^1(\mcO)$. 

\begin{theorem}\label{ra_s:thm:singleton_RA_general}
  Assume $(H')$ and let $\vp$ be a contractive RDS on $H$ such that $P_tf(x):= \E f(\vp(t,\cdot)x)$ is a weak-$*$ mean ergodic Markovian semigroup on $B_b(H)$. 
  Then for any random compact set $K(\o) \subseteq H$
    \[\text{diam}(\vp(t,\o)K(\o)) \to 0,\ \P-\text{a.s.} \]
  for $t \to \infty$. In particular, each invariant random compact set $K$ (i.e.\ $\vp(t,\o)K(\o)=K(\t_t\o)$) consists of a single random point.
\end{theorem}

\begin{corollary}\label{ra_s:cor:strongly_mixing}
  Under the assumption of Theorem \ref{ra_s:thm:singleton_RA_general} the semigroup $P_t$ is strongly mixing in the sense that for each $\nu \in \mcM_1$ we have $P_t^*\nu \to \mu$ weakly for $t \to \infty$.
\end{corollary}

In order to apply Theorem \ref{ra_s:thm:singleton_RA_general} to concrete applications we need a criterion for weak-$*$ mean ergodicity for singular SPDE. We will use the following assumptions in order to apply a result given in \cite{LT11}:
\begin{enumerate}
  \item [$(A')$] Assume that $A$ is independent of $(t,\o)$, $(A1)$-$(A5)$ are satisfied with $C_1 \equiv 0$ and $f$, $c$, $C_2$ being positive constants and that there exist $c > 0$, $\d \in (0,\a)$ such that
    \[ 2\Vbk{A(v_1)-A(v_2),v_1-v_2} \le -c \frac{\|v_1-v_2\|_H^2}{\|v_1\|_V^\d + \|v_2\|_V^\d},\]
  for all $v_1, v_2 \in V$.
  \item [$(S')$] Let $W_t=N_t$ be a $V$-valued Wiener process.
\end{enumerate}

%

\begin{corollary}[Singleton Random Attractors]\label{ra_s:thm:singleton_RA_add}
  Assume $(A')$, $(S')$, $(H')$. Then the random $\mcD^b$-attractor $\mcA^b$ obtained in Theorem \ref{ra_s:thm:ra_add} consists of a single random, fixed point, i.e.
    \[ \mcA(\o) = \{\eta(\o)\} \]
  and $\vp_t(\o)\eta(\o) = \eta(\t_t\o)$.
\end{corollary}

\subsection{Real linear multiplicative noise}

First we need to construct the associated RDS, which again will be defined by first transforming the SPDE into a random PDE and then solving this random PDE for each fixed $\o \in \O$. Let $X(t,s;\o)x$ denote a variational solution to \eqref{ra_s:eqn:SPDE_real_mult} starting in $x$ at time $s$. Define $\mu_t := e^{-\mu \b_t}$ and note that $\mu_t$ satisfies
  \[ d\mu_t = - \mu \mu_t \circ d\b_t,  \]
where $\circ$ is the Stratonovich stochastic integral. For $\tilde{X}(t,s;\o)x := \mu_tX(t,s;\o)x$ we obtain
\begin{align*}
  \tilde{X}(t,s;\o)x= \mu_s x + \int_s^t \mu_r A(r,{\mu}_r^{-1} \tilde{X}(r,s;\o)x) dr.
\end{align*}
Thus, we have to solve the following random PDE
\begin{align}\label{ra_s:eqn:transformed_PDE_mult}
   Z(t,s;\o)x = x + \int_s^t A_\o(r,Z(r,s;\o)x) dr,
\end{align}
with $A_\o(r,v) := \mu_r(\o) A (r,\mu_r^{-1}(\o)v)$. We then define the RDS associated to \eqref{ra_s:eqn:SPDE_real_mult} by
 \[ S(t,s;\o)x := \mu_t^{-1}(\o) Z(t,s;\o)(\mu_s(\o) x ).\]
In the following let $(\O,\mcF,\P,\t_t)$ be the metric dynamical system associated to two-sided real valued Brownian motion (cf. \cite{A98}). As in the case of additive noise we obtain
\begin{theorem}[Generation]\label{ra_s:thm:generation_mult}
  Assume $(A1)$-$(A4)$. Then, the family of mappings $S(t,s;\o)x$ is a continuous stochastic flow in $H$. In addition, $S(t,s;\o)x$ is continuous in $t$ and right-continuous in $s$. If $A$ is $(\mcB(\R) \otimes \mcB(V) \otimes \mcF, \mcB(V^*))$-measurable then $S(t,s;\o)x$ is a measurable stochastic flow. If $A(t,v;\o)$ is strictly stationary then $S(t,s;\o)x$ is a cocycle and hence $\vp(t,\o) := S(t,0;\o)$ is a continuous RDS. 
\end{theorem}

\begin{theorem}\label{ra_s:thm:finite_extinction}
  Assume that $A$ is $(\mcB(\R) \otimes \mcB(V) \otimes \mcF, \mcB(V^*))$-measurable, strictly stationary and satisfies $(A1)$-$(A4)$. Moreover, assume that there is a function $\l: \O \to \R_+\setminus\{0\}$ and a $0<p<2$ such that
  \begin{equation*}\label{ra_s:eqn:strict_coerc}
    \Vbk{A(t,v;\o),v} \le - \l(\o) \|v\|_H^p.
  \end{equation*}
  Then
   \[ \mcA(\o) := \{0\}\]
  is forward-absorbing in the sense that for every bounded set $B \subseteq H$, $s \in \R$ and $\o \in \O$ there is an absorption time $t_0 =t_0(\|B\|_H,s,\o)$ such that $\vp(t,\o)B \subseteq \{0\}$ for all $t \ge t_0$.

  If $A(t,0;\o)=0$ for all $t\in \R$, $\o\in\O$ then $\mcA$ is invariant under $\vp$ and thus $\mcA$ is a forward attractor for $\vp$.
\end{theorem}

\section{Applications}

In \cite[Lemma 3.1]{GLR11} it has been shown that for each $V$-valued process $N_t$ with stationary increments and a.s.\ c\`adl\`ag paths there is a metric dynamical system $(\O,\F,\P,\{\t_t\}_{t \in \R})$ and a version $\td N_t$ on $(\O,\F,\P,\{\t_t\}_{t \in \R})$ such that $\td N_t$ satisfies $(S1)$-$(S2)$. \\
Moreover, for any L\'evy process $N_t$ with L\'evy characteristics $(m,R,\nu)$ (e.g. cf. \cite[Corollary 4.59]{PZ07}) and $\int_V \left(\|x\|_V + \|x\|^2_V \right) d\nu(x) < \infty$, we have $\frac{N_t}{|t|} \rightarrow \pm \E N_1$ $\P$-almost surely for $|t| \rightarrow \infty$ (cf.\ \cite[Lemma 3.2]{GLR11}). In particular $\|N_t\|_V = O(|t|)$ for $|t| \to \infty$ and thus $(S3)$ is satisfied for every $\a \in (1,2)$. By splitting the L\'evy process $N_t$ into a L\'evy process with jump measure of bounded support and a compound Poisson process as suggested in \cite{B96}, the moment assumptions can be relaxed to $\int_V \left(\|x\|_V + \mathbbm{1}_{B_1(0)}(x)\|x\|^2_V \right) d\nu(x) < \infty$. In case of L\'evy processes on a Hilbert space $\int_V \mathbbm{1}_{B_1(0)}(x)\|x\|^2_V d\nu(x) < \infty$ is always satisfied and thus only finite first moment has to be assumed.

We now proceed to concrete examples of SPDE satisfying the assumptions $(A1)$-$(A5)$ and $(A')$. 
%
%

%

\subsection{Generalized Stochastic Singular \texorpdfstring{$p$}{p}-Laplace Equation}
  Let $(M,g,\nu)$ be a $d$-dimensional weighted compact smooth Riemannian manifold equipped with Riemannian metric $g$, associated measure $\mu$ and $d\nu(x) := \s(x) d\mu(x)$ with $\s$ being a smooth, positive function on $M$. Further, let $\a \in (1 \vee \frac{2d}{2+d},2)$ and $V := W_0^{1,\a}(M,\nu) \subseteq H := L^2(M,\nu)$. By the assumption on $\a$, the embedding $V \subseteq H$ is well-defined and compact. We denote the inner product on $T_xM$ given by the Riemannian metric $g$ by $(\cdot,\cdot)_x$ and the associated norm by $|\cdot|_x$. Let $N_t$ be a $V$-valued process satisfying $(S1)$-$(S3)$ on the metric dynamical system $\left(\O,\{\mcF_t\}_{t \in \R},\{\t_t\}_{t \in \R},\P\right)$. Consider the singular $p$-Laplace equation
  \begin{equation}\label{ra_s:eqn:SpLP}
     dX_t = \left( div_\nu(\Phi(x,\nabla X_t,\o)) + G(X_t,\o) + g(t,\o) \right)dt + dN_t(\o),
  \end{equation}
  where $\Phi: M \times TM \times \O \to TM$ is measurable, $\Phi(x,\cdot,\o):T_xM \to T_xM$ is continuous and
  \begin{equation*}\begin{split}
     (\Phi(x,\xi,\o)-\Phi(x,\td\xi,\o),\xi-\td \xi)_x &\le 0 \\
     (\Phi(x,\xi,\o),\xi)_x &\le c(\o)|\xi|_x^\a + f(\o)\\
     |\Phi(x,\xi,\o)|_x^\frac{\a}{\a-1} &\le C_2(\o) |\xi|_x^\a + f(\o), \quad \forall x \in M,\ \xi,\td\xi \in T_xM,\ \o \in \O,
  \end{split}\end{equation*}
  with $f: \O \to \R$ being measurable, $G:\R\times\O\to\R$ is measurable with
  \begin{equation*}\begin{split}
      |G(t,\o)-G(s,\o)|         &\le C_2(\o)|t-s| \\
      |G(t,\o)|^\frac{q}{q-1}   &\le C_2(\o)(1+|t|^q),\quad\forall t,s \in \R,\ \o \in \O,
  \end{split}\end{equation*}
  for some $q \in (1,\a)$ and $g: \R \times \O \to H$ is measurable, c\`adl\`ag in $t$. As an explicit example for an admissible reaction term one may consider $G(r)= \frac{r}{\sqrt{r^2+\ve}}$. 

  The singular $p$-Laplace operator then maps $V\times\O \to V^*$ by 
    $$A(v,\o)(w) = -\int_M (\Phi(x,\nabla v,\o),\nabla w)_x d\nu(x),\quad v,w \in V,\ \o \in \O.$$
  We obtain

 \begin{example}[Generalized Stochastic Singular $p$-Laplace Equation]\label{ra_s:exam:SpLE}
    There is an associated compact stochastic flow $S(t,s;\o)x$ to \eqref{ra_s:eqn:SpLP}. If $g(\cdot,\o) = o(|\cdot|^\frac{\a}{2-\a})$ then there is a measurable, random $\mcD^\a$-attractor $\mcA^\a \in \mcD^\a$. 

    If $g \equiv 0$ then $S(t,s;\o)x$ is a cocycle and there is a measurable, strictly stationary random $\mcD^b$-attractor $\mcA^b$.

    If $M \subseteq \R^d$ is an open, bounded set, $\nu = dx$, $\Phi(\xi)=|\xi|^{\a-2}\xi$, $N_t$ is a $V$-valued Wiener process and $G, g \equiv 0$, then the random attractor $\mcA^b$ consists of a single random fixed point, i.e.\ $\mcA^b(\o) = \{\eta(\o)\}$. 
   
    In case of real linear multiplicative noise
      \[ dX_t = div_\nu(\Phi(x,\nabla X_t)) dt + \mu X_t \circ d\b_t, \]
    the deterministic set $\mcA(\o) = \{0\}$ forward absorbs all bounded deterministic sets and is invariant.
\end{example}
\begin{proof}
   The proof of the properties $(A1)$-$(A4)$ proceeds as in \cite{PR07}. 
  $(A5)$ is satisfied by Sobolev embeddings and the assumption on $\a$. 

  In case of the standard nonlinearity $\Phi(x,\xi) = |\xi|_x^{\a-2}\xi$ and $G \equiv 0$ we can check $(A')$ as in \cite[Proposition 3.2]{LT11} with $\d = 2-\a < \a$. By monotonicity of $A$, $\vp$ is contractive on $H$.

  For simplicity we now restrict to the case of open, bounded domains $M \subseteq \R^d$, $\Phi(\xi)=|\xi|^{\a-2}\xi$, $N_t$ being a Wiener process in $V$ and $G, g \equiv 0$. In order to verify $(H')$ we set $S = H = L^2(M)$ and $x \le y$ for $x, y \in H$ iff $x(\xi) \le y(\xi)$ for almost all $\xi \in M$. Existence of upper bounds is obvious. It remains to prove monotonicity of $\vp$. We consider a non-singular approximation of the nonlinearity $\Phi(\xi) := |\xi|^{\a-2}\xi$ given by $\Phi^\ve(\xi) = \left( |\xi|^2 + \ve \right)^\frac{-(2-\a)}{2}\xi$. Then
    \[  |\Phi(\xi) - \Phi^\ve(\xi)| \le 2\ve^\frac{\a-1}{2} \]
  and as a composition of smooth functions, $\Phi^\ve$ is a smooth function. Let $\{e_n\}_{n \in \N} \subseteq C^\infty(M) \cap H_0^1(M)$ be an orthonormal basis of $H$, $H_n := $ span$\{e_1,...,e_n\}$ and $\mcP_n$ be the best-approximation by elements in $H_n$ weighted by $\|\cdot\|_V$, i.e.
    $$\|\mcP_n x - x\|_V = \inf_{v \in H_n}\|v-x\|_V,\quad \forall x \in V.$$
  Then $\|\mcP_nx\|_V \le C \|x\|_V$ and $\mcP_n x \to x$ in $V$ for $n \to \infty$ and $x \in V$ (cf.\ \cite{G11}). Define $N^n_t := \mcP_n N_t$. For initial conditions $x \in C^2(M)$ classical results (cf.\ e.g.\ \cite{LSU67}) imply the existence of a classical solution $Z^{\ve,n}$ to 
    \[ dZ^{\ve,n}_t = div\left(\Phi^\ve(\nabla (Z^{\ve,n}_t + N^n_t)) \right)dt.\]
  Since $\Phi^\ve$ is differentiable we can apply the comparison result given in \cite[Theorem 9.7]{L96} to obtain $Z^{\ve,n,x}_t \le Z^{\ve,n,y}_t$ on $[0,T]\times M$ for any two initial conditions $x \le y$ with $x,y \in C^2(M)$. Note
    \[  \|A(x) - A^\ve(x)\|_{V^*} = \sup_{\|v\|_V = 1} \int_\mcO \left(\Phi(\nabla x) - \Phi^\ve(\nabla x)\right) \cdot \nabla v d\xi \le C \ve^\frac{p-1}{2}. \]
  Since the operators $A^\ve$ satisfy uniform coercivity and growth conditions an application of Proposition \ref{ra_s:prop:convergence} yields $Z^{\ve,n} \to Z^n$ in $C([0,T];H)$ for $\ve \to 0$. By dominated convergence we have $N^n_\cdot(\o) \to N_\cdot(\o)$ in $L^\a([0,T];V)$ for each $\o \in \O$. This implies $Z^{n} \to Z$ in $C([0,T];H)$. Since ''$\le$`` is closed with respect to the $H$-norm, we obtain 
    $$Z_t^x \le Z_t^y, \text{ for all } t \in [0,T] \text{ and a.e.\ in }  M,$$
  for initial conditions $x \le y$, $x,y \in C^2(M)$. By continuity in the initial condition this extends to all $x \le y$, $x,y \in H$.
\end{proof}

\subsection{Generalized Stochastic Fast Diffusion Equation}

Let $(E,\mcB,m)$ be a finite measure space with countably generated $\s$-algebra $\mcB$ and let $(L,\mcD(L))$ be a negative-definite, self-adjoint, strictly coercive (i.e.\ $(-Lv,v)_{L^2(m)} \ge c\|v\|_{L^2(m)}^2$) operator on $L^2(m)$. Define $\mcD(\mcE) := \mcD(\sqrt{-L})$ and $\mcE(u,v) := (\sqrt{-L}u,\sqrt{-L}v), \text{ for } u,v \in \mcD(\mcE)$, where we have set $m(fg) := \int_E fg\ dm$, for $fg \in L^1(m)$, $(f,g) := (f,g)_{L^2(m)}$ and $\|f\| := \|f\|_{L^2(m)}$. Then $(D(\mcE),\mcE)$ is a Hilbert space. 

Let $\Phi: \R\times\O \to \R$ be measurable such that $\Phi(0,\o) = 0$, $\Phi(\cdot,\o) \in C(\R)$ and
\begin{equation}\label{ra_s:eqn:phi_ass_fde}\begin{split}
  (\Phi(r,\o)-\Phi(s,\o))(r-s)      &\ge 0\\
  \Phi(r,\o)r                         &\ge c(\o)|r|^\a - f(\o) \\
  |\Phi(r,\o)|^\frac{\a}{\a-1}        &\le C_2(\o)|r|^\a + f(\o), \quad \forall \o \in \O,\  s\le r,
\end{split}\end{equation}
for some $\a \in (1,2)$, $c: \O \to \R\setminus\{0\}$, $C_2: \O \to \R$ and $f: \O \to \R$ measurable. In particular, the standard nonlinearity $\Phi(r) := |r|^{\a-2}r$ is included in our general framework. We assume
\begin{enumerate}
 \item[(L)] The embedding $\mcD(E) \subseteq L^\frac{\a}{\a-1}(m)$ is compact and dense.
\end{enumerate}
This yields the Gelfand triple
  $$ V := L^{\a}(m) \subseteq  H := D(\mcE)^* \subseteq V^*.$$

\begin{example}
  Let
  \begin{enumerate}
    \item $E$ be a smooth, compact Riemannian $d$-dimensional manifold, $\a \in (1\vee\frac{2d}{d+2},2)$ and $L$ be the Friedrichs extension of a symmetric, uniformly elliptic operator of second order on $L^2(m)$ with Dirichlet boundary conditions. For example, let $L$ be the Dirichlet Laplacian on $E$.
    \item $E \subseteq \R^d$ be an open, bounded domain, $L:=(-\D)^\b$ with its standard domain and $\b \in \frac{d}{2}\left(\frac{2-\a}{\a},1\right) \cap (0,1]$.
  \end{enumerate}
  Then $(L)$ is satisfied.
\end{example}

Let $g: \R \times \O \to H$ be measurable with c\'adl\'ag paths and $g(t,\o)=o(|t|^\frac{\a}{2-\a})$ for $t \to -\infty$. 

\begin{example}[Generalized Stochastic Fast Diffusion Equation] The compact stochastic flow $S(t,s;\o)x$ associated to the stochastic fast diffusion equation
  \begin{equation}\label{ra_s:PME}
    dX_t= \left(L \Phi(X_t) + g(t)\right)dt + dN_t,
  \end{equation}
  with $N$ satisfying $(S1)$-$(S3)$ has a measurable random $\mcD^\a$-attractor $\mcA^\a \in \mcD^\a$.

  If $g \equiv 0$ then $S(t,s;\o)x$ is a cocycle and there is a measurable, strictly stationary random $\mcD^b$-attractor $\mcA^b$.

  If $E \subseteq \R^d$ is an open, bounded set, $L=\D$, $\Phi(r) = |r|^{\a-2}r$, $g\equiv 0$ and $N$ is a Wiener process in $V$
  then $\mcA^b$ is a single random point. 
  
  For real linear multiplicative noise
    $$  dX_t= L \Phi(X_t) dt + \mu X_t \circ d\b_t, $$
  and nonlinearites $\Phi$ satisfying \eqref{ra_s:eqn:phi_ass_fde} with $f \equiv 0$, $\mcA := \{0\}$ is invariant and forward absorbs all bounded sets $B \subseteq H$.
\end{example}
\begin{proof}
  The properties $(A1)$-$(A4)$ can be proven as in \cite{PR07}, $(A')$ with $\d = 2-\a$ as in \cite{LT11}. $(A5)$ is satisfied by assumption and monotonicity of $A$ implies contractivity of $S(t,s;\o)x$.

  For simplicity we now restrict to the case of $E \subseteq \R^d$ being an open, bounded set. Let $L=\D$, $\Phi(r) = |r|^{\a-2}r$ and $N$ be a Wiener process in $V$. Set $S = V$ and define $H_+ \subseteq H$ to be the closed, convex cone of all nonnegative distributions in $H$ with induced partial order structure ''$\le$'' on $H$. For elements $x,y \in V \subseteq H$ we have $x \le y$ iff $x(\xi) \le y(\xi)$ for almost all $\xi \in \mcO$. Existence of upper bounds in $S$ is obvious. It remains to prove monotonicity of $\vp$. For this we consider a smooth approximation of the nonlinearity $\Phi(r) = |r|^{\a-2}r$ given by $\Phi^\ve(r) = (|r|^2+\ve)^\frac{\a-2}{2}r$. Then $|\Phi(r)-\Phi^\ve(r)| \le 2\ve^\frac{\a-1}{2}$. Let $\{e_n\}_{n \in \N} \subseteq C^\infty(M) \cap H_0^1(M)$ be an orthonormal basis of $H$, $H_n := $ span$\{e_1,...,e_n\}$, $\mcP_n$ be the best-approximation by elements in $H_n$ weighted by $\|\cdot\|_V$ (cf. Example \ref{ra_s:exam:SpLE}) and  $N_t^n := \mcP_n N_t$. By classical existence results for uniformly parabolic quasilinear PDE (cf.\ \cite{LSU67}), the approximating equation
    \[ \frac{d}{dt} Z_t^{\ve,n} = \D\Phi^\ve(Z_t^{\ve,n}+N_t^n) \]
  has a unique classical solution for initial conditions in $C^2(E)$. By classical comparison results \cite[Theorem 9.7]{L96} for two such initial conditions $x \le y$, $x,y \in C^2(E)$ we obtain 
    \[ Z_t^{\ve,n,x} \le Z_t^{\ve,n,y}, \text{ on } [0,T] \times E.\]
  We conclude the proof as in Example \ref{ra_s:exam:SpLE}.
%
\end{proof}


\section{Proofs}\label{ra_s:sec:proofs}

\subsection{Stochastic flows and RDS}

First, we prove some properties of $\O$-limit sets of asymptotically compact stochastic flows. Similar results have been obtained in \cite{CLR06}.
\begin{lemma}\label{ra_s:lemma:suff_cond_attr}
  Let $S(t,s;\o)$ be a continuous stochastic flow.
  \begin{enumerate}
    \item[(i)] Assume that $S(t,s;\o)$ is $\mcD$-asymptotically compact. Then
      $$\O(D,t;\o) \subseteq K(t,\o) \cap \O(K,t;\o)$$
      is a compact, invariant set for all $D \in\mcD$, $t \in \R$, $\o \in \O_0$, where $\O_0$ is as in Definition \ref{ra_s:def:abs} and $\O(D,t;\o)$ attracts $D$. 
    \item[(ii)] If $S(t,s;\o)x$ is a cocycle and $D$ is strictly stationary, then $\O(D,t;\o)$ is strictly stationary.
    \item[(iii)] If $\{D(t,\o)\}_{t \in\R,\ \o \in \O}$ is a right lower-semicontinuous random closed set, $s \mapsto S(t,s;\o)x$ is right-continuous locally uniformly in $x$ and $S(t,s;\o)x$ is measurable, then $\O(D,t;\o)$ is a random closed set.
  \end{enumerate}
\end{lemma}
\begin{proof}
  (i): Since $S(t,s;\o)x$ is $\mcD$-asymptotically compact, there is a $\mcD$-attracting compact set $K$ on some subset $\O_0 \subseteq \O$ of full $\P$-measure. Since $K$ is $\mcD$-attracting we know $d(S(t,s;\o)D(s,\o),K(t,\o)) \to 0$ for $s \to - \infty$ for all $\o \in \O_0$, $t \in \R$. Hence, 
    $$d\left( \overline{\bigcup_{\tau \le r}S(t,\tau;\o)D(\tau,\o)},K(t,\o) \right) \to 0$$
  for $r \to -\infty$ and thus $d(\O(D,t;\o),K(t,\o)) = 0$, i.e.\ $\O(D,t;\o) \subseteq K(t,\o)$. 

  Next, we prove invariance of $\O(D,t;\o)$. Let $x \in \O(D,s;\o)$. There are sequences $s_n \to -\infty$, $x_n \in D(s_n,\o)$ such that $S(s,s_n;\o)x_n \to x$. By the flow property $S(t,s_n;\o)x_n = S(t,s;\o)S(s,s_n;\o)x_n \to S(t,s;\o)x \in \O(D,t;\o)$. \\
  Let now $z \in \O(D,t;\o)$, i.e.\ $S(t,s_n;\o)x_n \to z$ for some $s_n \to -\infty$ and $x_n \in D(s_n,\o)$. By $\mcD$-asymptotic compactness of $S(t,s;\o)x$ there is a subsequence $S(s,s_{n_l};\o)x_{n_l} \to x \ni  \O(D,s;\o)$. Hence, $S(t,s;\o)x = \lim_{l \to \infty} S(t,s;\o)S(s,s_{n_l};\o)x_{n_l} = z$.
  
  Invariance of $\O(D,t;\o)$ together with $\O(D,t;\o) \subseteq K(t,\o)$ then yields
    $$\O(D,t;\o) \subseteq \O(K,t;\o) \cap K(t,\o).$$

  Assume $\O(D,t;\o)$ does not attract $D$. Then there are $t \in \R$, $\ve > 0$, $\o \in \O_0$ and sequences $s_n \to -\infty$, $x_n \in D(s_n,\o)$ such that $d(S(t,s_n;\o)x_n,\O(D,t;\o)) \ge \ve$. By asymptotic compactness we can choose a convergent subsequence $S(t,s_{n_l};\o)x_{n_l} \to x \ni \O(D,t;\o)$ which leads to a contradiction. 

  (ii): To prove strict stationarity of $\O(D,t;\o)$ we note
    $$\O(D,t;\o) = \bigcap_{r \le t} \overline{\bigcup_{\tau \le r} S(t,\tau;\o)D(\tau,\o)} = \bigcap_{r \le t} \overline{\bigcup_{\tau \le r} S(0,\tau-t;\t_{t}\o)D(\tau-t,\t_{t}\o)} = \O(D,0;\t_{t}\o)$$

  (iii): $\O(D,t;\o)$ is a countable intersection of sets of the form $\overline{ \bigcup_{\tau < r} S(t,\tau;\o)D(\tau,\o)}$. 
  Hence, it is enough to prove measurability of $\o \mapsto d\left(x, \bigcup_{\tau < r} S(t,\tau;\o)D(\tau,\o) \right)$ for all $r \le t$, $x \in H$. Let $\tau_n \downarrow \tau$. By right lower-semicontinuity of $D(\cdot,\o)$ for each $y \in D(\tau,\o)$ there is a sequence $y_n \in D(\tau_n,\o)$ such that $y_n \to y$. Local uniform continuity of $\tau \mapsto S(t,\tau;\o)x$ thus yields $S(t,\tau_n;\o)y_n \to S(t,\tau;\o)y$. Hence
    $$ \limsup_{n \to \infty} d\left(x, S(t,\tau_n;\o)D(\tau_n,\o) \right) \le  d\left(x, S(t,\tau;\o)D(\tau,\o)\right)  $$
  and thus
  \begin{align*}
     \o \mapsto d\left(x, \bigcup_{\tau < r} S(t,\tau;\o)D(\tau,\o)\right) 
     &= \inf_{\tau < r}d\left(x, S(t,\tau;\o)D(\tau,\o) \right) \\
     &= \inf_{\tau < r,\ \tau \in \Q}d\left(x, S(t,\tau;\o)D(\tau,\o) \right).
  \end{align*}
  Therefore it is enough to prove measurability of $\o \mapsto d\left(x, S(t,\tau;\o)D(\tau,\o) \right)$ for all $\tau \le t$  which is satisfied by measurability of $S(t,\tau;\cdot)x$ and $D(\tau,\cdot)$ and by separability of $H$.
\end{proof}

\begin{proof}[Proof of Theorem \ref{ra_s:thm:suff_cond_attr}:]
  By  Lemma \ref{ra_s:lemma:suff_cond_attr} we know $\mcA(t,\o) \subseteq K(t,\o) \cap \O(K,t;\o)$ for all $\o \in \O_0$, $t \in \R$. In particular, $\mcA$ is compact. Since $\O(D,t;\o) \subseteq \mcA(t,\o)$ for all $\o \in \O_0$ and $\O(D,t;\o)$ is $D$ attracting, $\mcA$ is $\mcD$-attracting. Compactness of $\bigcup_{D \in \mcD} \O(D,t;\o)$ and invariance of $\O(D,t;\o)x$ yield invariance of $\mcA$.

  Let now $s \mapsto S(t,s;\o)x$ be right-continuous locally uniformly in $x$, $S(t,s;\o)x$ be measurable and (i) be satisfied. Then, by Lemma \ref{ra_s:lemma:suff_cond_attr}
    $$\mcA(t,\o) = \overline{\bigcup_{D_0 \in \mcD_0} \O(D_0,t;\o)}$$
  is the closure of a countable union of random closed sets. Hence, $\mcA$ is a random closed set. If (ii) holds, then $\O(K,t;\o) \subseteq \mcA(t,\o)$ and thus
    $$\mcA(t,\o) = \O(K,t;\o),$$
  for all $\o \in \O_0$, which is a closed random set by Lemma \ref{ra_s:lemma:suff_cond_attr}.
%
%
\end{proof}

\subsection{Generation of an RDS (Theorem \ref{ra_s:thm:generation_add})}

As outlined in Section \ref{ra_s:sec:main_result} we construct the stochastic flow associated to \eqref{ra_s:eqn:SPDE_add} by proving the unique existence of a solution to the transformed equation \eqref{ra_s:eqn:transformed_PDE_add} via the variational approach to (S)PDE as given in \cite{PR07}. To do so we check the assumptions $(H1)$-$(H4)$ in \cite{PR07} for $A_\o(t,v)$. For the ease of notation we suppress the $\o$-dependency of the coefficients occurring in the following calculations. $(H1)$, $(H2)$ immediately follow from $(A1)$, $(A2)$.
%

$(H3)$: For $v \in V$, $\o \in \O$ and $t \in \R$: 
\begin{align}\label{ra_s:eqn:H3-1}
  2{ }_{V^*}\< A_\o(t,v), v \>_V  
   &= 2 { }_{V^*}\< A \left(t,v + N_t \right),v+ N_t \>_V -2 { }_{V^*}\< A \left(t, v + N_t \right), N_t \>_V \nonumber\\
   &\le C_1(t) \|v + N_t\|_H^2 - c(t) \|v + N_t\|_V^\a + f(t) \\
     &\hskip10pt + 2 \|A \left(t, v + N_t \right)\|_{V^*} \|N_t\|_V \nonumber.
\end{align}
Using Young's inequality for all $\ve_1 > 0$ and some $C_{\ve_1}$  we obtain
\begin{flalign*}
  2 \|A \left(t, v + N_t \right)\|_{V^*} \|N_t\|_V  
  &\le \ve_1 \|A \left(t, v + N_t \right)\|_{V^*}^{\frac{\a}{\a-1}} + C_{\ve_1} \|N_t\|_V^\a  &&\\
  &\le \ve_1 C_1(t) \|v + N_t\|_H^2 + \ve_1 C_2(t) \|v + N_t\|_V^\alpha + \ve_1 f(t) + C_{\ve_1} \|N_t\|_V^\a,
\end{flalign*}
Using this in \eqref{ra_s:eqn:H3-1} yields
\begin{align*}
    2{  }_{V^*}\< A_\o(t,v), v \>_V  
   &\le C_1(t)(1+ \ve_1) \|v + N_t\|_H^2 - (c(t)-\ve_1 C_2(t)) \|v + N_t\|_V^\a + (1+\ve_1)f(t) \\
     &\hskip10pt  + C_{\ve_1} \|N_t\|_V^\a .
\end{align*}
Using
  \[ \|v+N_t(\o)\|_V^\a \ge  2^{1-\a}\|v\|_V^\a - \|N_t(\o)\|_V^\a\]
we obtain (for $\ve_1$ small enough):
\begin{align*}
  2{  }_{V^*}\< A_\o(t,v), v \>_V 
  &\le 2C_1(t)(1+ \ve_1 ) \|v\|_H^2 - (2^{1-\a}c(t)- \ve_1 2^{1-\a} C_2(t)) \|v\|_V^\a + (1+\ve_1)f(t)\\
     &\hskip10pt  + (C_{\ve_1}+c(t)-\ve_1 C_2(t))\|N_t\|_V^\a + 2C_1(t)(1+ \ve_1) \|N_t\|_H^2.
\end{align*}
which yields
\begin{align}\label{ra_s:eqn:transf_coerc}
  2\ _{V^*}\< A_\o(t,v), v \>_V  &\le \tilde{C}_1(t) \|v\|_H^2 - \tilde{c}(t) \|v\|_V^\a + \tilde{f}(t),
\end{align}
with 
\begin{align*}
  \tilde{f}(t) &:= (1+\ve_1)f(t)  + (C_{\ve_1}+c(t)-\ve_1 C_2(t)) \|N_t\|_V^\a + 2C_1(t)(1+ \ve_1) \|N_t\|_H^2 .\nonumber
\end{align*}
By right-continuity of $N_\cdot(\o)$ and by choosing $\ve_1$ small enough we obtain $(H3)$ for each compact interval $[S,T] \subseteq \R$. 

$(H4)$: For $v \in V$, $\o \in \O$ and $t \in \R$: 
\begin{align}\label{ra_s:eqn:H4-t}
  \|A_\o(t,v)\|_{V^*}^{\frac{\a}{\a-1}}g
  &=\| A \left(t, v + N_t \right) \|_{V^*}^{\frac{\a}{\a-1}} \nonumber \\
  &\le C_1(t)  \|v + N_t\|_H^2 + C_2(t)  \|v + N_t\|_V^\a + f(t) \\
  &\le \tilde{C}_1(t) \|v\|_H^2 + \tilde{C}_2(t) \|v\|_V^\a + \tilde{f}(t),  \nonumber
\end{align}
with $\tilde{C}_1(t) := 2 C_1(t), \tilde{C}_2(t) = 2^{\a-1} C_2(t)$ and
\begin{align*}
  \tilde{f}(t) 
    &:= f(t) + 2 C_1(t) \|N_t\|_H^2 +  2^{\a-1} C_2(t)\|N_t\|_V^\a.
\end{align*}

Hence $(H1)-(H4)$ are satisfied for $A_\o$ and by \cite[Theorem 4.2.4]{PR07} we obtain the unique existence of a solution 
\begin{equation*}
  Z(\cdot,s;\o)x \in L^\a_{loc}([s,\infty);V) \cap C([s,\infty);H) 
\end{equation*}
to \eqref{ra_s:eqn:transformed_PDE_add} for all $s \in \R$, $\o \in \O$, $x \in H$. By uniqueness for \eqref{ra_s:eqn:transformed_PDE_add} we have the flow property
\begin{equation*}
  Z(t,s;\o)x = Z(t,r;\o)Z(r,s;\o)x.
\end{equation*}
which implies that
\begin{equation*}
  S(t,s;\o)x := Z(t,s;\o)(x-N_s(\o))+N_t(\o)
\end{equation*}
defines a stochastic flow. 

The continuity of $t \mapsto Z(t,s;\o)x$ is contained in \cite[Theorem 4.2.4]{PR07}. Since $N_t(\o)$ is c\`adl\`ag in $t$ this implies that $t \mapsto S(t,s;\o)x$ is c\`adl\`ag. Monotonicity of $A_\o$ implies
  \[ \|Z(t,s;\o)x - Z(t,s;\o)y\|^2_H \le e^{\int_s^t C_2(r)dr}\|x-y\|_H^2. \]
Thus $x \mapsto Z(t,s;\o)x$ is continuous, uniformly in $t,s$ on bounded sets. Moreover,
\begin{align*}
  \|Z(t,s_1;\o)x - Z(t,s_2;\o)x\|^2_H 
  &= \|Z(t,s_2;\o)Z(s_2,s_1;\o)x - Z(t,s_2;\o)x\|^2_H  \\
  &\le e^{\int_{s_2}^t C_2(r)dr}\|Z(s_2,s_1;\o)x-x\|_H^2,\ \forall s_1 < s_2, 
\end{align*}
which implies right-continuity of $s \mapsto Z(t,s;\o)x$ and thus of $s \mapsto S(t,s;\o)x$ locally uniformly in $t$ and $x$.

Let now $A$ be $(\mcB(\R) \otimes \mcB(V) \otimes \mcF, \mcB(V^*))$-measurable. Then measurability of $Z(t,s;\o)x$ and $S(t,s;\o)x$ follows as in \cite[Theorem 1.1]{GLR11}.

Assume that $A(t,v;\o)$ is strictly stationary. We note
\begin{align*}
  Z(t,s;\o)x 
  &= x + \int_s^t A(r,Z(r,s;\o)x+N_r(\o),\o)dr \\
  &= x + \int_0^{t-s} A(r,Z(r+s,s;\o)x+N_s(\o)-N_0(\t_s\o)+N_r(\t_s\o),\t_s\o)dr.
\end{align*}
By uniqueness for \eqref{ra_s:eqn:transformed_PDE_add} we have
\begin{align*}
  Z(t,s;\o)x + N_s(\o) - N_0(\t_s\o)= Z(t-s,0;\t_s \o)(x+N_s(\o)-N_0(\t_s\o)).
\end{align*}
Hence
\begin{align*}
  S(t,s;\o)x  
    &=  Z(t,s;\o)(x-N_s(\o))+N_t(\o) \\
    &= Z(t-s,0;\t_s \o)(x-N_0(\t_s\o)) + N_t(\o)-N_s(\o)+N_0(\t_s\o) \\
    &= Z(t-s,0;\t_s\o)(x-N_0(\t_s\o)) + N_{t-s}(\t_s\o) \\
    &= S(t-s,0;\t_s\o)x,
\end{align*}
i.e.\ $S(t,s;\o)x$ is a cocycle.


\subsection{Compactness of the stochastic flow (Theorem \ref{ra_s:thm:compactness})}
\begin{proof}
  We will first show compactness of $Z(t,s;\o)x$. Let $\o \in \O$, $B \subseteq H$ bounded, $s < t$ and $z_n \in Z(t,s;\o)B$, i.e. $z_n = Z(t,s;\o)b_n$ for some sequence $b_n \in B$. We need to prove the existence of a convergent subsequence of $z_n$. First note that by \eqref{ra_s:eqn:transf_coerc}
  \begin{align*}
    \|Z(t,s;\o)b_n\|_H^2 
    &\le \|b_n\|_H^2 + \int_s^t \tilde{C}_1(r)\|Z(r,s;\o)b_n\|_H^2 dr \\
      &\hskip15pt - \int_s^t \tilde{c}(r) \|Z(r,s;\o)b_n\|_V^\a dr  + \int_s^t \tilde{f}(r)dr,
  \end{align*}
  where for notational convenience we do not explicitly write the $\o$-dependency of the coefficients. Since $\tilde{c}(r) > 0$ is right-continuous there is a $C > 0$ such that
    \[ \int_s^t \|Z(r,s;\o)b_n\|_V^\a dr \le C.\]
  By definition $Z(\cdot,s;\o)b_n$ satisfies
    \[ Z(r,s;\o)b_n = b_n + \int_s^r A_\o(\tau,Z(\tau,s;\o)b_n) d\tau,\]
  as an equation in $V^*$ for all $r \ge s$. Thus $\frac{d}{dr}Z(r,s;\o)b_n$ exists in $V^*$ (cf. \cite[Theorem 1.6., p.104]{S97}) and satisfies
    \[ \frac{d}{dr}Z(r,s;\o)b_n = A_\o(r,Z(r,s;\o)b_n) ,\]
  for almost all $r \in [s,\infty)$. For some constant $C > 0$ we obtain
  \begin{align*}
    &\int_s^t \left\| \frac{d}{dr}Z(r,s;\o)b_n \right\|_{V^*}^{\frac{\a}{\a-1}} dr = \int_s^t \|A_\o(r,Z(r,s;\o)b_n)\|_{V^*}^{\frac{\a}{\a-1}} dr \\
    &\le \int_s^t \tilde{C}_1(r) \|Z(r,s;\o)b_n\|_H^2 + \tilde{C}_2(r) \|Z(r,s;\o)b_n\|_V^\a + \tilde{f}(r) \ dr \\
    &\le \int_s^t \tilde{f}(r) dr + C < \infty,
  \end{align*}
  where the right hand side is independent of $n$. Thus $\{ Z(\cdot,s;\o)b_n \}$ is bounded in the space
  \begin{align*}
    W           &= \left\{ v \in L^\a([s,t];V), \frac{d}{dr}v \in L^{\frac{\a}{\a-1}}([s,t];V^*) \right\} \\
    \|v\|_W     &= \|v\|_{L^\a([s,t];V)}+ \left\|\frac{d}{dr}v \right\|_{L^{\frac{\a}{\a-1}}([s,t];V^*)}.
  \end{align*}
  By \cite[Theorem 2.1]{T01} $W \subseteq  L^\a([s,t];H)$ is compact. Hence $\{ Z(\cdot,s;\o)b_n \}$ is precompact in $L^\a([s,t];H)$ and we can choose a subsequence of $b_n$ (again denoted by $b_n$) and a $Z_0 \in L^\a([s,t];H)$ such that
    \[ Z(\cdot,s;\o)b_n \to Z_0,\]
  in $L^\a([s,t];H)$. Hence, by choosing a further subsequence of $b_n$ (denoting it by $b_n$ again) we obtain
    \[ Z(r,s;\o)b_n \to Z_0(r),\]
  in $H$, for almost every $r \in [s,t]$. Choose one such $r \in [s,t]$. Then
    \[ Z(t,s;\o)b_n = Z(t,r;\o)Z(r,s;\o)b_n \to Z(t,r;\o)Z_0(r).\]
  We have found the required convergent subsequence of $Z(t,s;\o)b_n$. Compactness of $Z(t,s;\o)x$ implies compactness of $S(t,s;\o)x$.
\end{proof}

\subsection{Comparison and a priori bounds}\label{ra_s:ssec:comp}
We present a comparison result and a priori bounds for certain ordinary differential equations that are the foundation of the proof of bounded absorption (Proposition \ref{ra_s:prop:add_bounded_absorption}).

\begin{lemma}[Comparison Lemma]\label{ra_s:lemma:comp_1}
  Let $0 < \b < 1$, $s \le t$, $q \ge 0$, $h \in L^1([s,t])$ nonnegative and $v: [s,t] \to \R_+$ be an absolutely continuous subsolution of
  \begin{align}\label{ra_s:eqn:comp_1_1}
    y'(r) &= -h(r)y(r)^\b,\ r \in [s,t],
  \end{align}
  with $y(s)  = q$, i.e. for almost every $r \in [s,t]$
  \begin{align}\label{ra_s:eqn:comp_1_2}
    v'(r) \le -h(r)v(r)^\b
  \end{align}
  and $v(s) \le q$. Then
    \[ v(r) \le \left( q^{1-\b} - (1-\b) \int_s^r h(\tau)d\tau \vee 0 \right)^{\frac{1}{1-\b}},\]
  for all $r \in [s,t]$.
\end{lemma}
\begin{proof}
  First note that since $h(r) \ge 0$, \eqref{ra_s:eqn:comp_1_1} is a monotone equation for $y(r) \ge 0$ and thus 
    $$y(r) := \left( q^{1-\b} - (1-\b) \int_s^r h(\tau)d\tau \vee 0 \right)^{\frac{1}{1-\b}}$$
  is the unique absolutely continuous nonnegative solution of \eqref{ra_s:eqn:comp_1_1}. Let $\ve >0$ and $y^\ve(r)$ be the unique nonnegative solution to \eqref{ra_s:eqn:comp_1_1} with $y^\ve(s) = q + \ve > 0$. Define $w^\ve(t) := y^\ve(t)-v(t)$ and
    \[ \tau^\ve := \inf\{r \in [s,t]\ |\ w^\ve(r) \le 0\} \wedge t= \inf\{r \in [s,t]\ |\ y^\ve(r) \le v(r)\} \wedge t.\]
  Since $y^\ve(s) = q + \ve > q \ge v(s)$ and $y^\ve, v$ are 
  continuous, we have $\tau^\ve > s$. If $\tau^\ve = t$ then nothing has to be shown. Thus suppose $\tau^\ve < t$. \\
  \textbf{Case 1:} $v(\tau^\ve) = 0$ \\
  By \eqref{ra_s:eqn:comp_1_2} $v(\cdot)$ is decreasing, hence $v(r) = 0 \le y^\ve(r)$ for all $r \in [\tau^\ve,t]$. Since also $v(r) \le y^\ve(r)$ for $r \in [s,\tau^\ve)$ this implies
    \[ v(r) \le y^\ve(r),\ \forall r \in [s,t].\]
  \textbf{Case 2:} $v(\tau^\ve) > 0$ \\
  Since $v(\cdot)$ is decreasing this implies the existence of a $\d >0$ such that $v(r) \ge \d > 0$ for $r \in [s,\tau^\ve]$. By definition of $\tau^\ve$, $v(r) \le y^\ve(r)$ on $[s,\tau^\ve)$. By the mean value theorem we further have
  \begin{align*}
    \left(w^\ve\right)'(r) \ge h(r) \left(-y^\ve(r)^\b + v(r)^\b\right) = -h(r)\b\xi_r^{\b-1} w^\ve(r),
  \end{align*}
  for almost every $r \in [s,t]$.
  We observe $ \xi_r^{\b-1} \le \d^{\b-1}$ for all $r \in [s,\tau^\ve]$. 
  Using Gronwall's inequality for absolutely continuous functions (cf.\ \cite[p.90]{T97}) 
   \[ w^\ve(r) \ge \ve e^{-\b \d^{\b-1} \int_s^r h(\tau)d\tau}, \text{ for all } r \in [s,\tau^\ve]. \]
  Hence
   \[ w^\ve(\tau^\ve) \ge \ve e^{-\b \d^{\b-1} \int_s^{\tau^\ve} h(\tau)d\tau} > 0,\]
  in contradiction to $w^\ve(\tau^\ve) \le 0$, by 
  continuity and definition of $\tau^\ve$. 

  Hence, the second case does not occur and we conclude $v(r) \le y^\ve(r),\ \forall r \in [s,t]$. Since this is true for all $\ve > 0$ we obtain $v(r) \le y(r),\ \forall r \in [s,t]$.
\end{proof}

\begin{lemma}[A-priori bound]\label{ra_s:lemma:comp_superlinear}
  Let $0 < \b < 1$, $0 < h$, $p: \R \to \R$ c\`adl\`ag, $q: \R \to \R_+$ and for each $s \in \R$ let $v(\cdot,s): [s,\infty) \to \R_+$ be an absolutely continuous subsolution of
  \begin{align}\label{ra_s:eqn:comp_2_1}
    y'(r,s) &= -h\ y(r,s)^\b + p(r),\ r \ge s
  \end{align}
  with $y(s,s) = q(s).$ We assume that $p(s)=o(|s|^\frac{\b}{1-\b})$ and $q(s)=o(|s|^{\frac{1}{1-\b}})$ for $s \to -\infty$, i.e.\ for each $\ve > 0$ there are $s_p(\ve), s_q(\ve)$ such that
  \begin{align*}
    |p(s)| &\le \ve |s|^{\frac{\b}{1-\b}}, \text{ for all } s \le s_p(\ve), \\
    q(s)   &\le \ve |s|^{\frac{1}{1-\b}}, \text{ for all } s \le s_q(\ve).
  \end{align*}
  Then for each $t \in \R$, there is an $s_0=s_0(t,s_q,h) \in \R$ and $R=R(t,p,s_p,h) > 0$ such that for all $s \le s_0$
    \[ v(t,s) \le R(t,p,s_p,h)\]
  and $R(t,p,s_p,h) = o(|t|^{\frac{1}{1-\b}})$ for $t \to -\infty$.
\end{lemma}
\begin{proof} 
  Without loss of generality we assume $p(r) \ge \d > 0$ (otherwise redefine $p(r):= p(r) \vee \d$). By scaling time by $\frac{1}{h}$ we can assume $h = 1$.

 Let $t\in \R$, $A(s) := \{ r \in [s,t] \ |\ \frac{1}{2} v(r,s)^\b \le p(r) \}$ and $a(s) = \sup A(s) \vee s$. We first show that there exists an $s_0=s_0(t,p,h) \le t$ such that $A(s) \ne \emptyset$ for all $s \le s_0$. Let $s \le t$ such that $A(s) = \emptyset$, 
  i.e. $ \frac{1}{2} v(r,s)^\b > p(r)$, for all $r \in [s,t]$. Hence, for almost every $r \in [s,t]$
  \begin{align*}
    v'(r,s) \le - v(r,s)^\b + p(r) \le - \frac{1}{2} v(r,s)^\b.
  \end{align*}
  By Lemma \ref{ra_s:lemma:comp_1} 
  \begin{align}\label{ra_s:eqn:apriori_1}
     0 < \big(2p(t) \big)^{\frac{1}{\b}} 
    &< v(t,s) 
    \le \left( q(s)^{1-\b} - \frac{1-\b}{2} (t-s) \vee 0 \right)^{\frac{1}{1-\b}}.
  \end{align} 
  For $\ve := \frac{ (1-\b)}{4}$ by assumption there exists an $s_q = s_q(\ve) \le 0$ such that $q(s)^{1-\b} \le 
 -\left(\frac{1-\b}{4}\right) s,$ for all $s \le s_q$. Hence
    \[ \left( q(s)^{1-\b} - \frac{1-\b}{2}(t-s) \vee 0  \right)^{\frac{1}{1-\b}} 
       \le \left(\frac{1-\b}{2}\right)^{\frac{1}{1-\b}} \left( \frac{s}{2} - t \vee 0 \right)^{\frac{1}{1-\b}} = 0, \]
  for $s \le s_q \wedge 2t$. Since also \eqref{ra_s:eqn:apriori_1} holds, we conclude $s \ge s_0:=s_q \wedge 2t$. Hence, for $s \le s_0$ we have $A(s) \ne \emptyset$.

  Next we prove that there exists an $a_1=a_1(t,s_p,h) \le t$ such that $a_1\le a(s) $ for all $s \le s_0$. Let $s \le s_0$, thus $A(s) \ne \emptyset$. If $a(s) =t$ then nothing is to show, thus suppose $a(s) < t$. By definition of $a(s)$ and right-continuity of $v,p$ we have 
   \[ p(r)\le\frac{1}{2} v(r,s)^\b , \text{ for all } r \in [a(s),t].\]
  Arguing as above we obtain
  \begin{align*}
    0 < \left( 2p(t) \right)^{\frac{1}{\b}} 
    &\le v(t,s) 
    \le \left(  v((a(s),s)^{1-\b} - \frac{1-\b}{2} (t-a(s)) \vee 0 \right)^{\frac{1}{1-\b}}
  \end{align*}
  Since $v$ is continuous and $p$ c\`adl\`ag we have $v(a(s),s) \le \big( 2 p(a(s)-) \big)^{\frac{1}{\b}}$. 
  For $\ve := \left(\frac{1}{2}\right)^{\frac{1-\b}{\b}} \left( \frac{1-\b}{4} \right)$ by assumption there exists an $s_p = s_p(\ve) \le 0$ such that  $p(s)^\frac{1-\b}{\b} \le 
     - \left(\frac{1}{2}\right)^{\frac{1-\b}{\b}} \left(\frac{1-\b}{4}\right) s,$ for all $s \le s_p$. Hence
    \[ \left(  v((a(s),s)^{1-\b} - \frac{1-\b}{2} (t-a(s)) \vee 0 \right)^{\frac{1}{1-\b}} \le \left(\frac{1-\b}{2}\right)^{\frac{1}{1-\b}} \left( \frac{a(s)}{2} - t \vee 0 \right)^{\frac{1}{1-\b}} = 0, \]
  if $a(s) \le s_p \wedge 2t$. Thus, we conclude $a_1 :=s_p \wedge t \wedge 2t \le a(s) $ for all $s \le s_0$.
  
  Since on $[a(s),t]$ we have $\frac{1}{2} v(r,s)^\b \ge p(r)$, we conclude for almost every $r \in [a(s),t]$
  \begin{align*}
    v'(t,s) \le - v(r,s)^\b + p(r) \le 0.
  \end{align*}
  Hence
  \begin{align*}
    v(t,s) \le v(a(s),s) \le \big( 2 p(a(s)-) \big)^{\frac{1}{\b}} \le \sup_{r \in [a_1-1,t]} \big( 2 p(r) \big)^{\frac{1}{\b}} =: R(t,p,h),
  \end{align*}
  for all $s \le s_0$. 
\end{proof}

\subsection{Bounded absorption (Proposition \ref{ra_s:prop:add_bounded_absorption})}
\begin{proof}
  We prove $\mcD^\a$-bounded absorption for $Z(t,s;\o)x$. By \eqref{ra_s:eqn:transf_coerc} we obtain
  \begin{align*}
    2\Vbk{A_\o(t,v), v}  &\le - \tilde{c} \|v\|_V^\a + \tilde{f}(t),
  \end{align*}
  with $\tilde{c} > 0$ and for some $C > 0$
  \begin{align*}
  \tilde{f}(t) &:= C\left( f(t)  + \|N_t\|_V^{\a} \right). 
  \end{align*}
  By the chain-rule
  \begin{align*}
    \frac{d}{dt}\|Z(t,s;\o)x\|_H^2 
    &= 2 \Vbk{A_\o(t,Z(t,s;\o)x),Z(t,s;\o)x} dr \\
    &\le - \tilde{c} (\|Z(t,s;\o)x\|_H^2)^{\frac{\a}{2}} + \tilde{f}(t),
  \end{align*}
  for almost all $t \in [s,\infty)$. Since $\|N_r\|_V = o(|r|^{\frac{1}{2-\a}})$ we know $\tilde{f}(r) = o(|r|^{\frac{\a}{2-\a}})= o(|r|^\frac{\b}{1-\b})$ for $r \to -\infty$ and $\b := \frac{\a}{2}$ for all $\o \in \O_0$. Let $D \in \mcD^\a$ and $x_s(\o) \in D_s(\o)$. We apply Lemma \ref{ra_s:lemma:comp_superlinear} with $v(t,s) := \|Z(t,s;\o)x_s(\o)\|_H^2$, $p(r) := \tilde{f}(r)$, $q(s) := |D_s(\o)|^2$, which is possible since $q(s) = o(|s|^\frac{1}{1-\b})$. Hence, for all $t \in \R$, $\o \in \O_0$ there is an absorption time $s_0=s_0(t,D,\o)$ and an $R=R(t,\o)$ such that
    \[ \|Z(t,s;\o)x_s(\o)\|_H^2 \le R(t,\o), \]
  for all $s \le s_0$. Since $s_0$ only depends on $q(s) := |D_s(\o)|^2$ this implies
    \[ \|Z(t,s;\o)D_s(\o)\|_H^2 \le R(t,\o), \]
  for all $s \le s_0$, i.e. $\mcD^\a$-absorption for $Z(t,s;\o)$
  . This implies
  \begin{align*}
    \|S(t,s;\o)D_s(\o)\|_H
    &= \|Z(t,s;\o)\left(D_s(\o) - N_s(\o) \right) + N_t(\o)\|_H \\
    &\le \|Z(t,s;\o)\left(D_s(\o) - N_s(\o) \right)\|_H + \|N_t(\o)\|_H \\
    &\le \sqrt{R(t,\o)} + \|N_t(\o)\|_H =: \td R(t,\o),
  \end{align*}
  for all $s \le \td s_0$, i.e. $\mcD^\a$-absorption for $S(t,s;\o)$ by the family of bounded sets 
  $$F(t,\o) :=
    \begin{cases}
      B(0,\td R(t,\o))&,\ \o\in \O_0 \\
      \{0\}&,\text{ otherwise.}
    \end{cases} $$

  The set $F$ is measurable iff $\o \mapsto R(t,\o)$ is measurable for each $t \in \R$. By the proof of Lemma \ref{ra_s:lemma:comp_superlinear} we have
    $$R(t,\o) = \sup_{r \in [a_1(t,\o)-1,t]} \left(2 \td f(r,\o) \right)^\frac{1}{\b}$$
  with $a_1(t,\o) = s_p \wedge t\wedge 2t$ and $s_p = s_p\left(\left(\frac{1}{2}\right)^\frac{1-\b}{\b}\frac{1-\b}{4}\right)$. Note
    $$ s_p(\ve) = \inf\{s\in\R_+ |\ \sup_{|r| \ge s} \frac{|p(r)|^\frac{1-\b}{\b}}{|r|} \le \ve\}.$$ 
  To prove measurability of $s_p(\ve)$ for $p = \td f$ we note
    $$\{s_p(\ve) < c\} = \bigcup_{n \in \N}\left\{\sup_{|r| \ge c-\frac{1}{n}}\frac{|\td f(r,\cdot)|^\frac{1-\b}{\b}}{|r|} \le \ve \right\} = \bigcup_{n \in \N}\left\{\sup_{|r| \ge c-\frac{1}{n},\ r \in \Q}\frac{|\td f(r,\cdot)|^\frac{1-\b}{\b}}{|r|} \le \ve \right\}.$$
  Measurability of $\td f$ thus implies measurability of $s_p(\ve)$. By right-continuity and measurability of $\td f$ the map
    $$ (s,\o) \mapsto \sup_{r \in [s,t]} \left(2 \td f(r,\o) \right)^\frac{1}{\b}$$
  is measurable in $\o$ and right-continuous in $s$. Hence, $R(t,\o)$ is measurable.

  Right lower-semicontinuity of $F$ is equivalent to $\limsup_{n \to \infty} R(t_n,\o) \ge R(t,\o)$ for each sequence $t_n \downarrow t$ which follows from right-continuity of $\td f(\cdot,\o)$ and $a_1 = s_p \wedge t \wedge 2t$. 

  By Lemma \ref{ra_s:lemma:comp_superlinear} $R(t,\o)=o(|t|^\frac{1}{1-\b})$ and thus $\td R(t,\o)=o(|t|^\frac{1}{2-\a})$ for $t \to -\infty$ and $\o \in \O_0$. This implies $F \in \mcD^\a$.
\end{proof}


\subsection{Existence of random attractors (Theorem \ref{ra_s:thm:ra_add})}

\begin{proof}[Proof of Theorem \ref{ra_s:thm:ra_add}: ]
  We prove compact absorption for $S(t,s;\o)$. Let $t \in \R$ and $\o \in \O$. By Proposition \ref{ra_s:prop:add_bounded_absorption} we know that there is a $\mcD^\a$ absorbing set $\{F(t,\o)\}_{t \in \R,\ \o \in \O} \in \mcD^\a$. Let
    \[ K(t,\o) := \overline{S(t,t-1;\o)F(t-1,\o)}. \]
  Since $S(t,s;\o)$ is a compact flow, $K(t,\o)$ is compact. Using \eqref{ra_s:eqn:transf_coerc}, $f(t) = o(|t|^\frac{\a}{2-\a})$ and $(S3)$ we observe $K \in \mcD^\a$. Furthermore $K(t,\o)$ is $\mcD^\a$-absorbing:
  \begin{align*}
    S(t,s;\o)D(s,\o) 
    &= S(t,t-1;\o)S(t-1,s;\o)D(s,\o) \\
    &\subseteq S(t,t-1;\o)F(t-1,\o) \subseteq K(t,\o),
  \end{align*}
  for $s \le s_0$ and $\o \in \O_0$. By Theorem \ref{ra_s:thm:suff_cond_attr} this yields the existence of a random $\mcD^\a$-attractor $\mcA^\a$ for $S(t,s;\o)x$ with $\mcA^\a(t,\o) \subseteq K(t,\o)$ for all $t \in \R$, $\o \in \O_0$. In particular, $\mcA^\a \in \mcD^\a$.

  If $A$ is measurable then $S(t,s;\o)x$ is a measurable stochastic flow. Since $s \mapsto S(t,s;\o)x$ is continuous locally uniformly in $t$ and $x$, right lower-semicontinuity of $F$ implies right lower-semicontinuity for $K$. Hence, by Theorem \ref{ra_s:thm:suff_cond_attr} $\mcA^\a$ is a random closed set.
  
  Now assume $A$ to be strictly stationary. Then $S(t,s;\o)x$ is a cocycle. The system of all bounded deterministic sets $\mcD^b$ satisfies condition (i) in Theorem \ref{ra_s:thm:suff_cond_attr}. Hence, there is a measurable, strictly stationary random $\mcD^b$-attractor $\mcA^b$.
\end{proof}


\subsection{Singleton Random Attractors (Theorem \ref{ra_s:thm:singleton_RA_general})}

The following Lemma is closely related to \cite[Proposition 1]{CS04}
\begin{lemma}\label{ra_s:lemma:mean_conv}
  Let $V$ be a Banach space with cone $V_+$ and $B \in \mcB(V)$. Let $X$, $Y$ be $(\mcB(\R_+)\otimes\mcF,\mcB(B))$-measurable $B$-valued processes such that $X_t \ge Y_t$ $\P$-almost surely. Assume that $\frac{1}{T} \int_0^T \mcL(X_t) dt \rightharpoonup \mu$ and  $\frac{1}{T} \int_0^T \mcL(Y_t) dt \rightharpoonup \mu$ weakly in $\mcM_1(B)$ where $\mu$ is a probability measure on $B$. Then 
    \[ \frac{1}{T} \int_0^T \P[\|X_t-Y_t\|_V \ge \d]dt \to 0\]
  for $T \to \infty$ and all $\d > 0$. If in addition $\|X_t-Y_t\|_V$ is non-increasing then $X_t-Y_t \to 0$ $\P$-almost surely.
\end{lemma}
\begin{proof}
  We can assume $B =V$ by extending $\mu$ by $0$ to all of $V$. Let $l \in V_+^*$ be a strictly positive linear functional, i.e.\ $l(v) > 0$ for $v \in V_+\setminus\{0\}$. We first prove 
    \[ \frac{1}{T} \int_0^T \P[l(X_t-Y_t) \ge \d]dt \to 0\]
  for $T \to \infty$ and for all $\d > 0$. It is sufficient to prove this in the case $\|l\|_{V^*} = 1$. Since $\frac{1}{T} \int_0^T \mcL(X_t) dt, \frac{1}{T} \int_0^T \mcL(Y_t) dt$ are weakly convergent, by Prokhorov's Theorem for each $\ve > 0$ there is an $N_\ve > 0$ such that 
    \[\frac{1}{T} \int_0^T \P[\|X_t\|_V \le N] dt  \wedge \frac{1}{T} \int_0^T \P[\|Y_t\|_V \le N] dt 
      \ge 1-\ve\]
  for all $N\ge N_\ve$ and all $T \ge 0$. Since $l(X_t) \le \|X_t\|_V$ we have $\P[l(X_t) \le N] \ge \P[\|X_t\|_V \le N]$. 
  Let 
    $$ F_N(r):= 
    \begin{cases}
        r  & \text{, for } |r| \le N \\
        sgn(r)N         & \text{, for } |r| > N.
    \end{cases} $$
  We observe
  \begin{align*}
    \frac{1}{T}\int_0^T \P[l(X_t-Y_t) \ge \d]dt 
    &= \frac{1}{T}\int_0^T \P[l(X_t)-l(Y_t) \ge \d]dt \\
    &= \frac{1}{T}\int_0^T \P[l(X_t)-l(Y_t) \ge \d, l(X_t) \vee l(Y_t) \le N]dt \\
      &\hskip15pt +  \frac{1}{T}\int_0^T \P[l(X_t)-l(Y_t) \ge \d, l(X_t) \vee l(Y_t) > N]dt\\
    &\le \frac{1}{T}\int_0^T \P[F_N \circ l(X_t)- F_N \circ l(Y_t) \ge \d, l(X_t) \vee l(Y_t) \le N]dt \\
      &\hskip15pt + \frac{1}{T}\int_0^T \P[l(X_t) > N]dt + \frac{1}{T}\int_0^T \P[l(Y_t) > N]dt \\
    &= \frac{1}{T}\int_0^T \frac{1}{\d} \E[F_N \circ l(X_t)]dt - \frac{1}{T}\int_0^T \frac{1}{\d} \E[F_N \circ l(Y_t)]dt + 2\ve 
    \le 3\ve,
  \end{align*}
  for $N \ge N_\ve$ and all $T \ge T_\ve$. By Prokhorov's Theorem for any $\ve > 0$ we can choose a compact set $\td K_\ve \subseteq V$ such that
    \[\frac{1}{T} \int_0^T \P[X_t \in \td K_\ve \text{ and } Y_t \in \td K_\ve] dt \ge 1-\ve,\]
  for all $T \ge 0$. Let $K_\ve = \left( \td K_\ve - \td K_\ve \right) \cap V_+$. Then
    \[\frac{1}{T} \int_0^T \P[X_t-Y_t \in K_\ve] dt \ge 1-\ve,\]
  for all $T \ge 0$. Hence, (cf.\ \cite[Lemma 1]{CS04}) there exists a strictly positive linear functional $l \in V^*$ such that
  \begin{align*}
    \frac{1}{T} \int_0^T \P[\|X_t-Y_t\|_V \ge \d] dt
    &\le \frac{1}{T} \int_0^T \P\left[ \frac{\d}{2} + C_{K_\ve,\frac{\d}{2}} l(X_t-Y_t) \ge \d \right]dt + \ve
      \le 2\ve,
  \end{align*}
  for all $T \ge T_\ve$.

  Let now $\|X_t-Y_t\|_V$ be non-increasing. We have
    \[ \frac{1}{T} \int_0^T \P[\|X_t-Y_t\|_V \ge \d]dt \to 0\]
  for $T \to \infty$ and all $\d > 0$. Let $\d > 0$. Then there is a sequence $T_n \to \infty$ such that $\P[\|X_{T_n}-Y_{T_n}\|_V \ge \d] \to 0$. Since $\|X_t-Y_t\|_V$ is non-increasing this implies $X_t-Y_t \to 0$ in probability and for the same reason $\P$-almost surely.
\end{proof}

\begin{proof}[Proof of Theorem \ref{ra_s:thm:singleton_RA_general}:]
  Let $\ve > 0$. Since $K(\o)$ is a random compact set, we can find a deterministic compact set $K_\ve$ such that
    \[ \P[K_\ve \supseteq K] \ge 1-\ve, \]
  by \cite[Proposition 2.15]{C02-2}. By compactness of $K_\ve$ and density of $S \subseteq H$ there is a finite $\ve$-net $K_1, ... , K_N \in S$ for $K_\ve$ (not necessarily $K_i \in K_\ve$). Since $\vp$ is contractive we have
    \[ \|\vp(t,\o)x-\vp(t,\o)y\|_H \le \|x-y\|_H, \]
  for all $t \in \R_+$, $\o \in \O$ and $x,y \in H$. This implies that $\vp(t,\o)K_1, ... , \vp(t,\o)K_N$ is an $\ve$-net for $\vp(t,\o)K_\ve$. By the existence of upper bounds in $S$, there is an upper bound $\bar K \in S$ satisfying $\bar K \ge K_i$ for all $i = 1,...,N$. Monotonicity of the RDS $\vp$ on $S$ yields $\vp(t,\o)\bar K \ge \vp(t,\o)K_i$. By weak-$*$ mean ergodicity of the associated Markovian semigroup and Lemma \ref{ra_s:lemma:mean_conv} this implies $\vp(t,\o)\bar K - \vp(t,\o)K_i \to 0$ $\P$-almost surely. Hence, for a.a.\ $\o \in \O$ there is a $t_{\ve,\o}$ such that $\|\vp(t,\o)\bar K(\o) - \vp(t,\o)K_i(\o)\| \le \ve$ for all $t \ge t_{\ve,\o}$, $i=1,...,N$. We conclude
  \begin{align*} 
    \text{diam}(\vp(t,\o)K_\ve) 
    &= \sup\{ \|\vp(t,\o)a-\vp(t,\o)b\|_H|\ a,b \in K_\ve\} \\
    &\le \sup\{ \|\vp(t,\o)K_i-\vp(t,\o)K_j\|_H|\ i,j = 1, ... ,N \} + 2\ve \\
    &\le 4\ve,
  \end{align*}
  for all $t \ge t_{\ve,\o}$, i.e. $\text{diam}(\vp(t,\o)K_\ve) \to 0$ almost surely. For any $\d>0$ we conclude
  \begin{align*} 
    \P[\text{diam}(\vp(t,\o)K(\o)) \ge \d] \le \ve + \P[\text{diam}(\vp(t,\o)K_\ve) \ge \d] \le 2 \ve,
  \end{align*}
  for all $t \ge t_\ve$, i.e. \text{diam}$(\vp(t,\o)K(\o)) \to 0$ in probability. By contractivity of $\vp$, \text{diam}$(\vp(t,\o)K(\o))$ is non-increasing. This implies \text{diam}$(\vp(t,\o)K(\o)) \to 0$ $\P$-almost surely.
\end{proof}

\begin{proof}[Proof of Corollary \ref{ra_s:cor:strongly_mixing}]
  It is sufficient to consider $\nu = \d_x$ for $x \in H$. Let $x \in H$, $f:H \to \R$ Lipschitz. Then 
   \begin{align*}
      |P_t f(x) - \mu(f)| 
      &= \left|\int_H \E f(\vp(t,\cdot)x) - f(\vp(t,\cdot)y) d\mu(y) \right| \\
      &\le \int_H \E |f(\vp(t,\cdot)x) - f(\vp(t,\cdot)y)| d\mu(y) \\
      &\le Lip(f) \int_H \E \|\vp(t,\cdot)x - \vp(t,\cdot)y\|_H d\mu(y).
   \end{align*}
   By Theorem \ref{ra_s:thm:singleton_RA_general} we have $\|\vp(t,\o)x-\vp(t,\o)y\| \to 0$ $\P$-almost surely for $t \to \infty$ and by contractivity $\|\vp(t,\cdot)x - \vp(t,\cdot)y\|_H \le \|x-y\|_H$. Since $\mu \in \mcM_1(H)$, dominated convergence yields
   \begin{align*}
     |P_t f(x) - \mu(f)| \to 0,
   \end{align*}
   for $t \to \infty$.  
\end{proof}

\begin{proof}[Proof of Corollary \ref{ra_s:thm:singleton_RA_add}]

Let
  \[ P_t f(x) = \E f(S(t,0;\cdot)x) = \E f(\vp(t,\cdot)x), \]
for $f \in B_b(H)$ and $x \in H$. The unique existence of variational solutions $X(t,0;\o)x$ is well known (cf.\ \cite[Theorem 4.2.4]{PR07}). By pathwise uniqueness of the solution, $S(t,0;\o)x$ and $X(t,0;\o)x$ are indistinguishable and thus the associated Markovian semigroups on $\mcB_b(H)$ coincide and are Feller (cf.\ \cite[Proposition 4.3.5]{PR07}). Monotonicity of the drift implies contractivity of $\vp$. By \cite[Theorem 1.3]{LT11} there is a unique invariant measure $\mu \in \mcM_1$ with $ \int_H \|x\|_V^\a d\mu(x) < \infty$ and for all Lipschitz functions $F: H \to \R$
  $$|P_tF(x)-P_tF(y)|\le \frac{C Lip(F) \|x-y\|_H}{\sqrt{t}}\left(1+ \frac{\|x\|_H}{\sqrt{t}}+ \frac{\|y\|_H}{\sqrt{t}} \right)^\frac{\d}{\a},$$
where $Lip(F)$ is the Lipschitz constant of $F$. In order to obtain bounds on higher moments of $\mu$ we note (as in \cite[Lemma 2.2]{LR10})
\begin{align*}
  \frac{1}{t}\E \int_0^t \|X(r,0;\cdot)0\|_H^k \|X(r,0;\cdot)0\|_V^\a dr \le C,
\end{align*}
for all $k \ge 4$, $t \ge 0$. Since $\mu$ is obtained as the weak limit of $\mu_n := \frac{1}{n}\int_0^n \mcL(X(r,0;\cdot)0)dr$ this implies
  $$ \int_H \|x\|_H^k \|x\|_V^\a d\mu(x) < \infty.  $$
In particular $\mu(\|\cdot\|_H^k) < \infty$ for all $k \ge 1$. We obtain
\begin{align*}
  |P_tF(x)- \mu(F)| 
  &\le \frac{C Lip(F) (\|x\|_H+1)}{\sqrt{t}} \left(1+ \frac{\|x\|_H^\frac{\d}{\a}}{t^\frac{\d}{2\a}}+ \frac{1}{t^\frac{\d}{2\a}} \right) \to 0,
\end{align*}
for $t \to \infty$ which implies weak$*$-mean ergodicity.
%
%
By Theorem \ref{ra_s:thm:ra_add} there is a measurable compact random attractor $\mcA$ and Theorem \ref{ra_s:thm:singleton_RA_general} yields $\text{diam}(\mcA(\t_t\o)) = \text{diam}(\vp(t,\o)\mcA(\o)) \to 0$. Since $\t_t$ is $\P$ preserving this implies $\text{diam}(\mcA(\o)) = 0$ almost surely, i.e.\ $\mcA(\o)$ is a single point.
\end{proof}


\subsection{Finite time extinction (Theorem \ref{ra_s:thm:finite_extinction})}
We will need the following 
\begin{lemma}\label{ra_s:lemma:integral_div}
  Let $\b$ be a real-valued Brownian motion and $q,s \in \R$. Then
    \[ \int_s^t e^{q (\b_r-\b_s)} dr \xrightarrow{t \to \infty} \infty,\]
  almost surely.
\end{lemma}
\begin{proof}
  We first prove convergence in probability. By the scaling property of Brownian motion we have $q (\b_{r+s}-\b_s) = q \b_{r} = 2 \b_{\frac{q^2}{4} r}$ in law. Hence,
    \[ \int_s^t e^{q (\b_r-\b_s)} dr = \int_0^{t-s} e^{q (\b_{r+s}-\b_s)} dr = \int_0^{t-s} e^{ 2 \b_{\frac{q^2}{4} r} } dr = \frac{4}{q^2} \int_0^{\frac{q^2}{4} ({t-s})} e^{ 2 \b_r},\]
  in law. Therefore, it is sufficient to prove divergence of
    \[ A_t := \int_0^{t} e^{ 2 \b_r},\]
  in probability.  By \cite{AMS01} we know that $A_t$ has the density function
    \[ \P[A_t \in du] = \frac{du}{\sqrt{2 \pi u^3}} \frac{1}{\sqrt{2\pi t}} \int_\R \cosh(\xi)e^{-\frac{\cosh(\xi)^2}{2u}-\frac{(\xi+\frac{i \pi}{2})^2}{2t}}d\xi. \] 
  Let $K > 0$. Then
    \begin{align*}
       \P[A_t \le K] 
       &= \int_\R \int_0^K \frac{1}{\sqrt{2 \pi u^3}} \frac{1}{\sqrt{2\pi t}}  \cosh(\xi)e^{-\frac{\cosh(\xi)^2}{2u}-\frac{(\xi+\frac{i \pi}{2})^2}{2t}} du d\xi.
    \end{align*}
    Taking the absolute value and substituting $u \rightsquigarrow \frac{\cosh(\xi)}{\sqrt{u}}$ in the inner integral yields
    \begin{align*}
       \P[A_t \le K] 
       &\le 
       2 e^{\frac{\pi^2}{8t}} \int_\R \int_{\frac{\cosh(\xi)}{\sqrt{K}}}^\infty \frac{e^\frac{-u^2}{2}}{\sqrt{2 \pi}}  \frac{e^{-\frac{\xi^2}{2t}}}{\sqrt{2\pi t}}  du d\xi.
    \end{align*}
    Let now $\ve > 0$ and choose $R >0$ such that $\int_{\frac{\cosh(R)}{\sqrt{K}}}^\infty \frac{e^\frac{-u^2}{2}}{\sqrt{2 \pi}} du \le \frac{\ve}{8}. $ Then choose $t$ so large that $\sup_{\xi \in [-R,R]} \frac{e^{-\frac{\xi^2}{2t}}}{\sqrt{2\pi t}} \le \frac{\ve}{16 R}$ and $2 e^{\frac{\pi^2}{8t}} \le 4$. We obtain
    \begin{align*}
       \P[A_t \le K] 
       &=  2 e^{\frac{\pi^2}{8t}} \left( \int_{-R}^R \int_{\frac{\cosh(\xi)}{\sqrt{K}}}^\infty \frac{e^\frac{-u^2}{2}}{\sqrt{2 \pi}}  \frac{e^{-\frac{\xi^2}{2t}}}{\sqrt{2\pi t}}  du d\xi +  \int_{[-R,R]^c} \int_{\frac{\cosh(\xi)}{\sqrt{K}}}^\infty \frac{e^\frac{-u^2}{2}}{\sqrt{2 \pi}}  \frac{e^{-\frac{\xi^2}{2t}}}{\sqrt{2\pi t}}  du d\xi \right) \\
       &\le 4 \left(\frac{\ve}{8} \int_{\frac{\cosh(1)}{\sqrt{K}}}^\infty \frac{e^\frac{-u^2}{2}}{\sqrt{2 \pi}} du +  \frac{\ve}{8}\int_{[-R,R]^c} \frac{e^{-\frac{\xi^2}{2t}}}{\sqrt{2\pi t}} d\xi \right) \le \ve,
    \end{align*}
    for all $t$ sufficiently large, i.e.\ $\P[A_t \le K]  \to 0,$ for $t \to \infty$. Hence, the same holds for $\int_s^t e^{q (\b_r-\b_s)} dr $. We conclude that for almost all $\o \in \O$ there exists a $t_0$ (depending on $\o$) such that $\int_s^{t_0} e^{q (\b_r-\b_s)} dr  > K$. Since $\int_s^{t} e^{q (\b_r-\b_s)} dr $ is increasing in $t$, this implies $\int_s^t e^{q (\b_r-\b_s)} dr  > K$ for all $t \ge t_0$, i.e.
      \[ \int_s^t e^{q (\b_r-\b_s)} dr  \to \infty, \]
    almost surely.
\end{proof}

\begin{proof}[Proof of Theorem \ref{ra_s:thm:finite_extinction}]:
  Let $B \subseteq H$ bounded and $x \in B$. Recall that $Z(t,s;\o)x$ satisfies 
  \begin{equation*}
   Z(t,s;\o)x = x + \int_s^t \mu_r(\o) A \left(r,\mu_r^{-1}(\o) Z(r,s;\o)x \right)dr.
  \end{equation*}
  In particular $Z(t,s;\o)x$ is absolutely continuous. By the chain-rule
  \begin{align*}
    \frac{d}{dt} \|Z(t,s;\o)x \|_H^2 
    &= \mu_t(\o)^2 \Vbk{A \left(t,\mu_t^{-1}(\o) Z(t,s;\o)x \right),\mu_t^{-1}(\o)Z(r,s;\o)x}\\
    &\le \l(\o) \mu_t(\o)^{2-p} \|Z(t,s;\o)x\|_H^p dr,
  \end{align*}
  for almost every $t \in [s,\infty)$. Lemma \ref{ra_s:lemma:comp_1} yields
  \begin{align*}
    \| Z(t,s;\o)x \|_H^2 
    &\le \left( \|x\|_H^{2-p} - (1-\frac{p}{2})\l(\o) \int_s^t \mu_\tau(\o)^{2-p} d\tau \vee 0 \right)^{\frac{2}{2-p}},
  \end{align*} 
  for all $t \ge s$. Since $S(t,s;\o)x = \mu_t(\o) Z(t,s;\o)(\mu_s(\o)^{-1}x)$, we obtain
  \begin{align*}
    \| S(t,s;\o)x \|_H^2 
    &\le \mu_t(\o)^2 \left( \left(\|\mu_s(\o)^{-1}x\|_H^{2-p} - (1-\frac{p}{2})\l(\o) \int_s^t \mu_\tau(\o)^{2-p} d\tau\right) \vee 0 \right)^{\frac{2}{2-p}} \\
    &\le \mu_t(\o)^2 \mu_s(\o)^{-2} \left(\left( \|B\|_H^{2-p} - (1-\frac{p}{2})\l(\o) \int_s^t e^{-\mu(2-p)(\b_\tau(\o)-\b_s(\o))} d\tau\right) \vee 0 \right)^{\frac{2}{2-p}}.
  \end{align*} 
  Now we can apply Lemma \ref{ra_s:lemma:integral_div} to obtain $\| S(t,s;\o)x \|_H^2 = 0$, for $t \ge t_0$ and $t_0=t_0(\|B\|_H,\o,s)$ large enough. 
\end{proof}

\section{Appendix}
\begin{proposition}\label{ra_s:prop:convergence}
  Let $A^\ve: [0,T] \times V \to V^*$, $\ve \ge 0$ be a family of monotone operators converging pointwisely in $V^*$, i.e.\ for every $t \in [0,T]$ and  $x \in V$ we have 
    \[ \|A^\ve(t,x)-A^0(t,x)\|_{V^*} \to 0, \]
  for $\ve \to 0$. Assume that there exists an $\a > 1$ such that $\|A^\ve(t,x)\|_{V^*}^\frac{\a}{\a-1} \le C (\|x\|_V^\a + 1)$ and that $Y^\ve$ are variational solutions to the corresponding equations
    \[ \frac{d}{dt} Y_t^\ve + A^\ve(t,Y_t^\ve) = 0,\ (\ve \ge 0)\]
  satisfying the uniform bound $\|Y_r^\ve\|_{L^\a([0,T];V)} \le C.$ Then 
    \[ \sup_{t \in [0,T]} \|Y^\ve_t - Y^0_t\|_H^2 \to 0, \]
  for $\ve \to 0$.
\end{proposition}
\begin{proof}
  By the chain-rule
  \begin{align*}
    &\|Y_t^\ve-Y_t^0\|_H^2 = \int_0^t \dualdel{V}{-A^\ve(r,Y^\ve_r) + A^0(r,Y^0_r)}{Y_r^\ve-Y_r^0} dr \\
    &= \int_0^t \dualdel{V}{-A^\ve(r,Y^\ve_r) + A^\ve(r,Y^0_r)}{Y_r^\ve-Y_r^0} dr + \int_0^t \dualdel{V}{A^0(r,Y^0_r)-A^\ve(r,Y^0_r)}{Y_r^\ve-Y_r^0} dr \\
    &\le \int_0^t \dualdel{V}{A^0(r,Y^0_r)-A^\ve(r,Y^0_r)}{Y_r^\ve-Y_r^0} dr \\
    &\le \|A^0(r,Y^0_r)-A^\ve(r,Y^0_r)\|_{L^\frac{\a}{\a-1}([0,T];V^*)} \|Y^\ve-Y^0\|_{L^\a([0,T];V)}  \\
    &\le 2C \|A^0(\cdot,Y^0_\cdot)-A^\ve(\cdot,Y^0_\cdot)\|_{L^\frac{\a}{\a-1}([0,T];V^*)}.
  \end{align*}
  By dominated convergence the claim follows.
\end{proof}


\bibliographystyle{plain}
\bibliography{../latex-refs/refs}

\end{document}

%% file: preamble.tex
\newcommand{\E}{\mathbb{E}}

\newcounter{cprop}[section]

\newcommand{\mcB}{\mathcal{B}}

\newtheorem{definition}[cprop]{Definition}
\newtheorem{remark}[cprop]{Remark}
\newtheorem{lemma}[cprop]{Lemma}
\newtheorem{example}[cprop]{Example}
\newtheorem{proposition}[cprop]{Proposition}
\newtheorem{theorem}[cprop]{Theorem}
\newtheorem{corollary}[cprop]{Corollary}